\documentclass[amssymb, 11pt]{amsart}
\usepackage{latexsym}

\newcommand{\CC}{\mathbb{C}}
\newcommand{\QC}{\mathcal{Q}}
\newcommand{\GC}{\mathcal{G}}
\newcommand{\HC}{\mathcal{H}}
\newcommand{\F}{\mathbb{F}}
\newcommand{\p}{\pi}
\newcommand{\e}{\epsilon}

\newcommand{\ran}{{\boldsymbol{\rho}}^{1}_{n}}
\newcommand{\rbn}{{\boldsymbol{\rho}}^{2}_{n}}
\newcommand{\bve}{{\boldsymbol{\varepsilon}}}
\newcommand{\al}{\alpha}
\newcommand{\gam}{\gamma}
\newcommand{\omn}{{\boldsymbol{\omega}}_{n}}
\newcommand{\omk}{{\boldsymbol{\omega}}_{k}}
\newcommand{\oml}{{\boldsymbol{\omega}}_{l}}
\newcommand{\bj}{\boldsymbol{{\it j}}}
\newcommand{\bja}{\boldsymbol{{\it j}}_{1}}
\newcommand{\bjb}{\boldsymbol{{\it j}}_{2}}
\newcommand{\bl}{\boldsymbol{\lambda}}
\newcommand{\bln}{\bl_{n}}
\newcommand{\blk}{\bl_{k}}
\newcommand{\bll}{\bl_{l}}
\newcommand{\bL}{\boldsymbol{\Lambda}}
\newcommand{\bLn}{\bL_{n}}
\newcommand{\bd}{\boldsymbol{\delta}}
\newcommand{\bda}{\boldsymbol{\delta}_{1}}
\newcommand{\bdb}{\boldsymbol{\delta}_{2}}
\newcommand{\bal}{{\boldsymbol{\nu}}_{1}}
\newcommand{\bbe}{{\boldsymbol{\nu}}_{2}}
\newcommand{\bmu}{\boldsymbol{\mu}}
\newcommand{\taun}{{\boldsymbol{\tau}}_{n}}
\newcommand{\taui}{{\boldsymbol{\tau}}^{i}_{n}}
\newcommand{\DA}{D_{\alpha}}
\newcommand{\Irr}{\mathrm{Irr}}
\newcommand{\Ind}{\mathrm{Ind}}
\newcommand{\Ker}{\mathrm{Ker}}

\newlength{\standardunitlength}
\newtheorem{prop}{Proposition}[section]
\newtheorem{definition}[prop]{Definition}

\newtheorem{lemma}[prop]{Lemma}
\newtheorem{cor}[prop]{Corollary}
\newtheorem{theorem}[prop]{Theorem}

\begin{document}

\title [Cycle indices for finite orthogonal groups] {Cycle indices for finite orthogonal groups of even characteristic}

\author{Jason Fulman}
\address{Department of Mathematics\\
        University of Southern California\\
        Los Angeles, CA, 90089, USA}
\email{fulman@usc.edu}

\author{Jan Saxl}
\address{University of Cambridge \\ Cambridge CB3 0WA, UK}
\email{J.Saxl@dpmms.cam.ac.uk}

\author{Pham Huu Tiep}
\address{University of Arizona \\ Tucson, AZ, 85721-0089, USA}
\email{tiep@math.arizona.edu}


\date{Version of April 8, 2010}

\thanks{The authors are grateful to Martin Liebeck for kindly sending
them \cite{LS} which plays an important role in the current
paper.}

\thanks{Fulman was partially supported by NSF grant DMS-0802082 and NSA
grant H98230-08-1-0133. Tiep was partially supported by NSF grant
DMS-0901241.}

\begin{abstract} We develop cycle index generating functions for orthogonal groups in
even characteristic, and give some enumerative applications. A key step is the determination
of the values of the complex linear-Weil characters
of the finite symplectic group, and their inductions to the general linear
group, at unipotent elements.
We also define and study several natural probability measures on integer partitions.
\end{abstract}

\maketitle

\begin{center}
{\sl Dedicated to Peter M. Neumann on the occasion of his
seventieth birthday}
\end{center}

\section{Introduction}

        P\'{o}lya \cite{Po}, in a landmark paper on combinatorics (see \cite{Polya} for an
English translation), introduced the cycle index of the symmetric groups. This can be written as
follows. Let $a_i(\pi)$ be the number of $i$-cycles of $\pi$. The Taylor
expansion of $e^z$ and the fact that there are $n!/\prod_i (a_i! i^{a_i})
$
elements of $S_n$ with $a_i$ $i$-cycles, yield the following theorem.

\begin{theorem} (P\'{o}lya \cite{Po})

\[ 1 + \sum_{n=1}^{\infty} \frac{u^n}{n!} \sum_{\pi \in S_n} \prod_i
x_i^{a_i(\pi)} = \prod_{m=1}^{\infty} e^{\frac{x_mu^m}{m}} \]

\end{theorem}

        The P\'{o}lya cycle index has been a key tool in understanding what a
typical permutation $\pi \in S_n$ ``looks like''. It is useful for studying
properties of a permutation which depend only on its cycle structure. Here
are a few examples of theorems which can be proved using the cycle
index. Shepp and Lloyd \cite{She} showed that for any $i< \infty$, the
joint distribution of $(a_1(\pi),\cdots,a_i(\pi))$ for $\pi$ chosen
uniformly in $S_n$ converges to independent (Poisson($1$), $\cdots$,
Poisson($\frac{1}{i}$)) random variables as $n \rightarrow \infty$. Goncharov \cite{Gon}
proved that the number of cycles in a random permutation is asymptotically
normal with mean and variance $\log (n)$. Goh and Schmutz \cite{Go1} proved that
if $\mu_n$ is the average order of an element of $S_n$, then

\[ \log (\mu_n) = C \sqrt{\frac{n}{\log (n)}} (1+o(1)), \] where $C = 2.99047...$.

Given the above facts, it is very natural to seek cycle indices for finite classical
groups. Kung \cite{Kung} and Stong \cite{St} developed cycle indices for the tower of groups
$GL_{n}(q)$; applications, and extensions to $GU_{n}(q)$ and odd characteristic symplectic and
orthogonal groups appear in \cite{F1}. The paper \cite{W2} independently uses generating
function methods to study various proportions in $GL_{n}(q)$, and the memoir \cite{FNP} extends results
in \cite{F1} and \cite{W2} to other finite classical groups. Britnell \cite{B1}, \cite{B2},
\cite{B3}, \cite{B4} extends cycle index techniques to $SL_{n}(q)$, $SU_{n}(q)$, and odd characteristic
groups related to the finite symplectic and orthogonal groups. The case of even characteristic
symplectic groups was treated in \cite{FG1}, using representation theory.

These cycle indices for finite classical groups are quite useful; they have applications in computational
group theory \cite{NP1}, and were fundamental
to the proof of the Boston-Shalev conjecture that the proportion of derangements in a primitive action of
a simple group on a set $X$ with $|X|>1$ is uniformly bounded away from 0 (see \cite{FG2} and the references
therein). Even quite complicated statistics such as the order of a random matrix can be studied using cycle
index techniques \cite{Sc}; Schmutz's results along these lines were crucially applied by Shalev in \cite{Sh}.

The purpose of this paper is to obtain results for one important remaining case: even characteristic orthogonal groups.
Throughout this paper $O^{\pm}(n,q)$ denotes the full orthogonal group (not the conformal group), though some
authors use the notation $GO^{\pm}(n,q)$.
Since odd dimensional orthogonal groups are isomorphic to symplectic groups (and
one can easily move between the corresponding rational canonical forms in $GL_{2n}(q)$ and
$GL_{2n+1}(q)$, (see Lemma \ref{sp-O}), we assume that the dimension is even.
In principle the cycle indices could be obtained by adding conjugacy classes sizes of $O^{\pm}_{2n}(q)$ with a given $GL$ rational canonical form, and using formulas of Wall \cite{W1}. However this seems quite a daunting task, and Wall's treatment of conjugacy classes in even
characteristic finite orthogonal groups is so complicated that experts (Liebeck and Seitz) have initiated a program of revisiting Wall's work
(see for instance \cite{LS}) and Lusztig (\cite{L1}, \cite{L2}, \cite{L3})
has three recent papers on the topic. As another example of the complexity of the characteristic two case, see Andrews' proof \cite{A2} of the
Lusztig-Macdonald-Wall conjectures on enumerating conjugacy classes in $O^{\pm}(2n,q)$. Our strategy for studying characteristic two cycle indices employs representation theory, and a crucial step is the derivation of a formula for the complex linear-Weil characters
of $Sp_{2n}(q)$ on unipotent elements.

This paper is organized as follows. Section \ref{characters} performs the needed character theory calculations. This includes
several intermediate results such as a branching formula and parameterizations of unipotent classes which may be of
independent interest. Section \ref{odddim} briefly treats odd dimensional orthogonal groups, and
Section \ref{cycleindex} develops the cycle indices for $O^{\pm}(2n,q)$ and $\Omega^{\pm}(2n,q)$. Some enumerative
applications are given in Section \ref{enumerate}. Section \ref{partitions} defines and studies several probability measures on
integer partitions, which we speculate may arise as an orthogonal analog of the Cohen-Lenstra number field/function field heuristics.

\section{Character theory calculations} \label{characters}

\subsection{Some permutation characters} \label{permchar}
To begin we review some representation theory. Let
$S := Sp_{2n}(q)$ be the finite symplectic group stabilizing a non-degenerate
symplectic form $(\cdot ,\cdot)$ on $V = \F_{q}^{2n}$, with $q$ a power of 2.
Then $S$ acts as a permutation group on the set of quadratic forms polarized to
$(\cdot,\cdot)$. There are two orbits, depending on the Witt index (or the type)
of the forms. The two permutation characters, $\p ^{+}$ and $\p ^{-}$, are both
multiplicity-free. This is well known; a proof appears in \cite{nick},
and we sketch it below.
We use the methods of \cite{nick} to obtain a decomposition of these characters
into irreducible characters. We show first that $\p = \p^+ + \p^-$ is equal to the
permutation character $\taun$ of $S$
acting on the set of vectors of $V$.
We then show that
\begin{equation}\label{for-p}
  \p^{+} = 1_{S} + \rbn +  \sum^{(q-2)/2}_{i=1}\taui,~~~
  \p^{-} = 1_{S} + \ran +  \sum^{(q-2)/2}_{i=1}\taui,
\end{equation}
where $1_{S} + \ran + \rbn$ is the permutation character of the well-known
rank $3$ permutation action of $S$ on the set of $1$-subspaces of $V$,
and each of the $\taui$ is an irreducible character of degree $(q^{2n}-1)/(q-1)$.
These characters $\taui$ are restrictions of irreducible characters
of the corresponding general linear group $G = GL_{2n}(q)$ and are the complex linear-Weil
characters, investigated by Guralnick and Tiep \cite{gt}. We will only use
the fact that for $g$ unipotent \begin{equation} \label{val} \taui(g) = \frac{q^{d(g)}-1}{q-1},
\end{equation} for all $i$, where $d(g)$ is the dimension of the kernel of $g-1$.

Let $g \in S$. Then $\p (g) > 0$; the proof of this, due to Inglis,
appears in \cite[Lemma 4.1]{ss}. Let $Q$ be a quadratic form supported by
$(\cdot,\cdot)$, and fixed by $g$. If $R$ is a quadratic form supported
by $(\cdot,\cdot)$, then $Q+R$ is a quadratic form on $V$ which is totally defective
(that is, $Q+R$ is a quadratic form supported by the zero
symplectic form). Any such quadratic form is the square of a unique
linear functional $f_{Q,R}$ on $V$, and $R$ is fixed by $g$ if and only if
$f_{Q,R}$ is fixed by $g$. It follows that $\p$ equals the permutation
character $\taun$ of $S$ on the set of vectors of $V$.

Inglis takes this further: for bilinear forms
$Q$ and $R$ supported by $(\cdot,\cdot)$, let $y_{Q,R}$ be the
unique vector such that $(Q+R)(x) = (x, y_{Q,R})^2$ for all $x \in V,$
and then define $a(Q,R) = Q(y_{Q,R}) = R(y_{Q,R})$.
The pairs $(Q,R)$ and  $(Q_1,R_1)$ lie in the same
orbit of $S$ on ordered pairs of forms if and only if
$a(Q,R) = a(Q_1,R_1)$. From this it follows that
the permutation rank of the action of $S$ on the quadratic
forms of a given type $+$ or $-$ is $(q+2)/2$.
Since $a(Q,R) = a(R,Q)$, it follows
that each orbital in these actions of $S$ is self-paired, whence
the permutation characters $\p^{+}$ and $\p^{-}$ are both multiplicity
free. This last claim can also be seen directly: in dimension two
it is a very easy computation, and for $Sp_{2n}(q)$ it is seen by
restriction to $Sp_2(q^n)$ (note that $Sp_2(q^n)$ is transitive
in our actions of $Sp_{2n}(q)$).

It is shown in \cite[\S3]{gt} that
$$\taun = 2 \cdot 1_{S} + \ran + \rbn + 2\sum^{(q-2)/2}_{i=1}\taui.$$
Now the claimed decomposition (\ref{for-p})
for $\p^+$ and $\p^-$ easily follows
since the characters $\p ^{+}, \p ^{-}$ are multiplicity free,
each with $(q+2)/2$ constituents.

\subsection{Branching rules for linear-Weil characters}
We recall the construction \cite{T} of the dual pair
$Sp_{2n}(q) \times O^{+}_{2}(q)$
in characteristic $2$. Let $U = \F_{q}^{2n}$ be
endowed with standard symplectic form $(\cdot,\cdot)$. We will also
consider the $\F_{2}$-symplectic form
$<u,v> = {\mathrm{tr}}_{\F_{q}/\F_{2}}((u,v))$ on $U$, and let
$$E = \CC^{q^{2n}} = \langle e_{u} \mid u \in U \rangle_{\CC}.$$
Clearly, $S := Sp_{2n}(q)$ acts on $E$ via $g~:~e_{u} \mapsto e_{g(u)}$. Fix
$\delta \in \F_{q}^{\times}$ of order $q-1$, and consider the following
endomorphisms of $E$:
$$\bd~:~e_{u} \mapsto e_{\delta u}$$
(for any $u \in U$), and
$$\bj~:~e_{0} \mapsto e_{0},~~~~e_{v} \mapsto
  \frac{1}{q^{n}+1}\sum_{0 \neq w \in U,~<v,w> = 0}e_{w} -
  \frac{q^{n}+2}{q^{n}(q^{n}+1)}\sum_{w \in U,~<w,v> \neq 0}e_{w}$$
(for any $0 \neq v \in U$). One can check that
$D := \langle \bd,\bj \rangle \simeq O^{+}_{2}(q)$ (a dihedral group
of order $2(q-1)$), and that $D$
centralizes $S$. The subgroup $S \times D$ of $GL(E)$
is the desired dual pair $Sp_{2n}(q) \times O^{+}_{2}(q)$. Let $\omn$
denote the character of $S \times D$ acting on $E$. It is shown in
\cite{T} that $\omn|_{S} = \taun$, the permutation character of $S$ on the
point set of its natural module $V = \F_{q}^{2n}$. Moreover one can label the
irreducible characters of $D$ as $\bbe = 1_{D}$, $\bal$ of degree $1$,
and $\bmu_{i}$, $1 \leq i \leq (q-2)/2$, of degree $2$ such that
%\marginpar{forSD}
\begin{equation}\label{forSD}
  \omn|_{S \times D} = (\ran + 1_{S}) \otimes \bal +
  (\rbn + 1_{S}) \otimes \bbe + \sum^{(q-2)/2}_{i=1}\taui \otimes \bmu_{i}.
\end{equation}

We can repeat the above construction but with $n = k+l$ replaced throughout
by $k > 0$, resp. by $l > 0$, and subscript $1$, resp. $2$, attached to
all letters $U$, $E$, $S$, $D$, $\bd$, and $\bj$. Thus we get the dual pair
$S_{1} \times D_{1} \simeq Sp_{2k}(q) \times O^{+}_{2}(q)$ inside $GL(E_{1})$
with character $\omk$, and the dual pair
$S_{2} \times D_{2} \simeq Sp_{2l}(q) \times O^{+}_{2}(q)$ inside $GL(E_{2})$
with character $\oml$. Now we can identify $U$ with $U_{1} \oplus U_{2}$.
This in turn identifies $E$ with $E_{1} \otimes E_{2}$ and $\bd$ with
$\bda \otimes \bdb$. This identification also embeds $S_{1} \otimes S_{2}$
in $S$. In what follows, we denote $x_{1} := \bda^{a}\bja^{b}$ and
$x_{2} := \bdb^{a}\bjb^{b}$ for $x = \bd^{a}\bj^{b}$.

%\marginpar{mult1}
\begin{lemma}\label{mult1}
{ Let $Sp_{2k}(q) \times Sp_{2l}(q)$ be a standard subgroup of
$Sp_{2n}(q)$. Then
$\omn(gx) = \omk(g_{1}x_{1}) \cdot \oml(g_{2}x_{2})$ for any
$x \in O^{+}_{2}(q)$ and any
$g = g_{1} \otimes g_{2} \in Sp_{2k}(q) \times Sp_{2l}(q)$.}
\end{lemma}

\begin{proof}
Suppose first that $x = \bd^{a}$. Then
$$gx = (g_{1} \otimes g_{2})(\bda \otimes \bdb)^{a} =
  (g_{1} \otimes g_{2})(\bda^{a} \otimes \bdb^{a}) =
  g_{1}\bda^{a} \otimes g_{2}\bdb^{a} =
  g_{1}x_{1} \otimes g_{2}x_{2},$$
whence the statement follows by taking trace.

It remains to consider the case $x = \bd^{a}\bj$. Since all the elements
$\bd^{a}\bj$ for $0 \leq a < q-1$ are conjugate in
$D$, we may assume $a = 0$. Let
$$N := |\{w \in U \mid <w,g(w)> = 0\}|,~~~
  N_{i} :=  |\{w \in U_{i} \mid <w,g_{i}(w)> = 0\}|$$
for $i = 1,2$. One can check that
$$\omn(gx) = 2q^{-n}N -q^{n},~~~\omk(g_{1}x_{1}) = 2q^{-k}N_{1} -q^{k},~~~
  \oml(g_{2}x_{2}) = 2q^{-l}N_{2} -q^{l}.$$
To relate $N$ to $N_{1}$ and $N_{2}$, write $w = u_{1} + u_{2}$ for
$w \in W$ and $u_{i} \in U_{i}$. Then $g_{i}(u_{i}) \in U_{i}$ and so
$$<w,g(w)> = <u_{1}+u_{2},g_{1}(u_{1})+g_{2}(u_{2})> =
  <u_{1},g_{1}(u_{1})> + <u_{2},g_{2}(u_{2})>.$$
It follows that $<w,g(w)> = 0$ if and only if
$$<u_{1},g_{1}(u_{1})> = <u_{2},g_{2}(u_{2})> = 0 \mbox{ or }
  <u_{1},g_{1}(u_{1})> = <u_{2},g_{2}(u_{2})> = 1.$$
Hence $N = N_{1}N_{2}+(q^{2k}-N_{1})(q^{2l}-N_{2})$, and so
\begin{eqnarray*} \omn(gx) & = & 2q^{-n}(N_{1}N_{2}+(q^{2k}-N_{1})(q^{2l}-N_{2}))-q^{n} \\
& = &
 (2q^{-k}N_{1} -q^{k})(2q^{-l}N_{2} -q^{l}), \end{eqnarray*}
as stated.
\end{proof}

A well-known consequence of orthogonality relations (see e.g. Lemma 5.5 of
\cite{LBST}) implies that, for any $g \in S$ and $x \in D$,
\begin{equation}\label{dual}
  \omn(gx) = \sum_{\al \in~\Irr(D)}\al(x) \cdot \DA(g),
\end{equation}
where
$$\DA(g) = \frac{1}{|D|}\sum_{x \in D}\overline{\al(x)}\omn(gx).$$
We will use the decomposition (\ref{dual})
and Lemma \ref{mult1} to prove the following
branching rule for the virtual character $\bln:= \p^+-\p^-= \rbn-\ran$,
see (\ref{for-p}).

%\marginpar{mult2}
\begin{lemma}\label{mult2}
{ Let $Sp_{2k}(q) \times Sp_{2l}(q)$ be a standard subgroup of
$Sp_{2n}(q)$. Then
$\bln(g) = \blk(g_{1}) \cdot \bll(g_{2})$ for any
$g = g_{1} \otimes g_{2} \in Sp_{2k}(q) \times Sp_{2l}(q)$.}
\end{lemma}

\begin{proof}
We will use the notation introduced before Lemma \ref{mult1}. Applying
(\ref{dual}) to the dual pairs $S_{1} \times D$ and $S_{2} \times D$, we can
also write
$$\omk(g_{1}x_{1}) = \sum_{\al \in~\Irr(D)}\al(x_{1}) \cdot E_{\al}(g_{1}),~~~
  \oml(g_{2}x_{2}) = \sum_{\al \in~\Irr(D)}\al(x_{2}) \cdot F_{\al}(g_{2}),$$
where $E_{\al}$, resp. $F_{\al}$, plays the role of $\DA$ for $S_{1}$, resp.
for $S_{2}$, and $g = g_{1} \otimes g_{2}$.
By Lemma \ref{mult1}, we now have
$$\omn(gx) = \omk(g_{1}x_{1}) \cdot \oml(g_{2}x_{2}) =
   \sum_{\beta,\gam \in~\Irr(D)}\beta(x_{1})\gam(x_{2}) \cdot E_{\beta}(g_{1})
   F_{\gam}(g_{2}).$$
It follows that
$$\DA(g)  = \frac{1}{|D|}\sum_{x \in D,~\beta,\gam \in \Irr(D)}
  \overline{\al(x)}\beta(x_{1})\gam(x_{2})E_{\beta}(g_{1})F_{\gam}(g_{2})$$
$$= \sum_{\beta,\gam}
    \left( \frac{1}{|D|}\sum_{x \in D}\overline{\al(x)}\beta(x)\gam(x)\right)
    E_{\beta}(g_{1})F_{\gam}(g_{2}) =
    \sum_{\beta,\gam}[\beta\gam,\al]E_{\beta}(g_{1})F_{\gam}(g_{2}),$$
where $[\cdot,\cdot]$ is the usual scalar product on the space of class
functions on $D$. We will apply this identity to the cases where
$\al \in \{\bal,\bbe\}$. In these cases, $\al$ is linear; furthermore,
any $\beta \in \Irr(D)$ is real. Hence
$[\beta\gam,\al] \neq 0$ if and only if $\beta = \al\gam$, which means
that $\beta = \gam$ if $\al = \bbe = 1_{D}$. If $\al = \bal$, the
unique non-principal linear irreducible character of $D$, then
$\al\gam$ equals $\gam$, resp. $\bbe$, or $\bal$, if $\gam = \bmu_{i}$,
resp. $\gam = \bal$, or $\gam = \bbe$. We have therefore shown that
$$(D_{\bal})|_{S_{1} \times S_{2}} =
  E_{\bal} \otimes F_{\bbe} + E_{\bbe} \otimes F_{\bal} +
  \sum^{(q-2)/2}_{i=1}E_{\bmu_{i}} \otimes F_{\bmu_{i}},$$
$$(D_{\bbe})|_{S_{1} \times S_{2}} =
  E_{\bal} \otimes F_{\bal} + E_{\bbe} \otimes F_{\bbe} +
  \sum^{(q-2)/2}_{i=1}E_{\bmu_{i}} \otimes F_{\bmu_{i}}.$$
On the other hand, by (\ref{forSD}) we have
$D_{\bal} = \ran + 1_{S}$ and $D_{\bbe} = \rbn + 1_{S}$; in particular,
$\bln = D_{\bbe} - D_{\bal}$, and similarly
$\blk = E_{\bbe} - E_{\bal}$ and $\bll = F_{\bbe} - F_{\bal}$. Hence the
statement follows.
\end{proof}

\subsection{Homogeneous unipotent elements}
In this subsection we consider unipotent elements of $S = Sp_{2n}(q)$
which are {\it homogeneous}, i.e. its Jordan canonical form on
$V = \F_{q}^{2n}$ contains only Jordan blocks of the same size. Recall
that we fix an $S$-invariant nondegenerate symplectic form $(\cdot,\cdot)$
on $V$. Let $J_{a}$ denote the $a \times a$ Jordan block with eigenvalue $1$.
We say that $g \in S$ is {\it decomposable}, if $V$ can be written as an
orthogonal sum of nonzero $g$-invariant subspaces, and {\it indecomposable}
otherwise.

\begin{lemma}\label{homog}
{ Assume that the Jordan canonical form of $g \in Sp_{2n}(q)$ on
$V = \F_{q}^{2n}$ is $kJ_{a} = J_{a} \oplus \ldots \oplus J_{a}$ with
$2n = ka$ and $a \geq 2$. If $a = 2$, assume in addition that
$g$ is indecomposable. Then one of the following holds.

{\rm (i)} $n \geq 2$, all the $g$-invariant quadratic forms on
$V$ which are polarized to $(\cdot,\cdot)$ have the same type $\e = \pm$,
and $\bln(g) = \e q^{k}$. Moreover, if $a = 2$ then $k = 2$, and
$\e = +$.

{\rm (ii)} $n = k = 1$, $a = 2$, $\bln(g) = 0$.}
\end{lemma}

\begin{proof}
1) Consider a basis
$(e_{1}, \ldots ,e_{a},f_{1}, \ldots ,f_{a}, \ldots ,h_{1}, \ldots ,h_{a})$,
in which $g$ is represented by the matrix $kJ_{a}$. Let $\QC$ be the set of
all $g$-invariant quadratic forms on $V$ which are polarized to
$(\cdot,\cdot)$. Then $\QC \neq \emptyset$ and in fact
$|\QC| = \taun(g) = q^{k}$ as mentioned above. By \cite[Lemma 6.10]{Spa},
if $Q \in \QC$ then
$$Q(e_{i}) = (e_{i},e_{i+1}), ~~~Q(f_{i}) = (f_{i},f_{i+1}), \ldots ,
  Q(h_{i}) = (h_{i},h_{i+1})$$
for $1 \leq i \leq a-1$. Thus $Q$ is completely determined by
the $k$-tuple $(Q(e_{a}), \ldots ,Q(h_{a})) \in \F_{q}^{k}$.

2) Suppose that $a \geq 3$ and some $Q \in \QC$ has type $+$.
Then we can find a symplectic basis $(u_{1}, \ldots ,u_{n},v_{1}, \ldots ,v_{n})$
of $V$ such that $Q(u_{i}) = Q(v_{i}) = 0$. Since $a \geq 3$,
by \cite[Lemma 6.10]{Spa} the subspace
$W := \langle e_{1},f_{1}, \ldots ,h_{1}\rangle_{\F_{q}}$ is totally singular
with respect to any $Q' \in \QC$, in particular with respect to $Q$;
moreover,
$$W^{\perp} = \langle e_{1}, \ldots,e_{a-1},f_{1}, \ldots,f_{a-1}, \ldots ,
  h_{1}, \ldots,h_{a-1} \rangle_{\F_{q}}.$$
Notice that $U := \langle u_{1},u_{2}, \ldots ,u_{a}\rangle_{\F_{q}}$ is
totally $Q$-singular of the same dimension $a$ as of $W$. Hence
by Witt's Theorem, $W = \varphi(U)$ and $W^{\perp} = \varphi(U^{\perp})$
for some $\varphi \in S$ which preserves $Q$. But
$U^{\perp}$ contains the $n$-dimensional totally $Q$-singular subspace
$M := \langle u_{1},u_{2}, \ldots ,u_{n}\rangle_{\F_{q}}$. It follows that
$W^{\perp}$ contains the $n$-dimensional totally $Q$-singular subspace
$\varphi(M)$. Now consider any $Q' \in \QC$. As mentioned in 1),
$Q'$ and $Q$ coincide on $W^{\perp}$. Hence $\varphi(M)$ is also
totally singular with respect to $Q'$ and so $Q'$ is of type $+$.

We have shown that all $Q \in \QC$ have the same type $\e = \pm$. Recall
that $\bln(g)$ is the difference between the number of $Q \in \QC$ of
type $+$ and the number of $Q \in \QC$ of type $-$. It follows that
$\bln(g) = \e q^{k}$.

3) Next we consider the case $a = 2$ and $k \geq 2$. If $(e_{1},e_{2}) \neq 0$,
then $E := \langle e_{1},e_{2}\rangle_{\F_{q}}$ is $g$-invariant and
non-degenerate and so $V = E \oplus E^{\perp}$, contradicting the assumption
that $g$ is indecomposable. Thus
$(e_{1},e_{2}) = (f_{1},f_{2}) = \ldots = (h_{1},h_{2}) = 0$. As mentioned in 2),
$e_{1}$ is orthogonal to all the vectors $e_{1}, \ldots ,h_{1}$. Since
$(\cdot,\cdot)$ is non-degenerate, we may assume that
$(e_{1},f_{2}) = b \neq 0$. Then again by \cite[Lemma 6.10]{Spa},
$(e_{2},f_{1}) = b$. One can now check that
$F := \langle e_{1},e_{2},f_{1},f_{2}\rangle_{\F_{q}}$ is $g$-invariant and
non-degenerate. By the assumption that $g$ is indecomposable, we must have
$k = 2$. Furthermore, $E$ is totally singular with respect to any
$Q \in \QC$. Now we can finish the argument as in 2).

Finally, assume $a = 2$ and $k = 1$. Then $\ran = 0$ and $\rbn$ is just the
Steinberg character (of degree $q$) of $S \simeq SL_{2}(q)$, whence
$\bln(g) = 0$.
\end{proof}

Of course if $a = 1$ in Lemma \ref{homog} then $\bln(g) = q^{k/2}$.
For a general unipotent element $u \in S$, it is well known
(see e.g. \cite[Cor. 6.12]{Spa}) that $V$ can be written as
an orthogonal sum of (possibly zero) $u$-invariant subspaces
$V = \oplus^{\infty}_{i=1}V_{i}$ such that all Jordan blocks of $u|_{V_{i}}$
are of size $i$. Hence, combining Lemmas \ref{mult2} and \ref{homog}
we can compute $\bln(u)$ for any unipotent element $u \in S$. To evaluate
$\bLn  = \Ind^{GL_{2n}(q)}_{Sp_{2n}(q)}(\bln)$ at $u$ we however need more
information about unipotent classes in $S$ with a given Jordan canonical form.
This will be done in the next subsection, which also is of independent
interest.

\subsection{A parametrization of unipotent classes in finite symplectic
and orthogonal groups in characteristic $2$}
The conjugacy classes of finite classical groups
are described in \cite{W1}. A new, better, treatment of
this topic, particularly in bad characteristics,
has recently been given in \cite{LS}. We will use the latter results to
give a parametrization of unipotent classes in finite symplectic
and orthogonal groups in characteristic $2$ which works well for our
purposes and which also has independent interest.

Let $2|q$ as above and let $g \in S = Sp_{2n}(q)$ be
any unipotent element.
It is shown in \cite{LS} that the natural
$S$-module $V = \F_{q}^{2n}$ can be written as an orthogonal sum
\begin{equation}\label{dec1}
  V|_{\langle g \rangle} =
  \sum_{i}W(m_{i})^{a_{i}} \oplus \sum^{r}_{j=1}V(2k_{j})^{b_{j}}
\end{equation}
of $g$-invariant non-degenerate subspaces of two types: $V(m)$ with even $m$,
on which $g$ has the Jordan form $J_{m}$, and $W(m)$, on which $g$ is
indecomposable and has the Jordan form $2J_{m}$; moreover,
\begin{equation}\label{dec2}
  m_{1} < m_{2} < m_{3} < \ldots,~~~k_{1} < k_{2} < k_{3} < \ldots, ~~~
  a_{i} > 0,~~~2 \geq b_{j} \geq 1.
\end{equation}
Given such a decomposition (\ref{dec1}) subject to (\ref{dec2})
(called a {\it canonical decomposition} in \cite{LS}),
let $I$ be the set of indices $i$ such that $m_{i}$ is odd and larger than
$1$ and there does {\bf not} exist $j$ such that $2k_{j} = m_{i} \pm 1$, and
let $s := |I|$.
Next, let $t$ be the number of indices $j$ such that $k_{j+1}-k_{j} \geq 2$.
We also fix $\delta \in \{0,1\}$ with $\delta = 1$ precisely when
$r > 0$ and $k_{1} > 1$. We can view $S$ as the fixed point subgroup
$\GC^{F}$ for a Frobenius endomorphism $F$ on
$\GC = Sp_{2n}(\overline{\F}_{q})$.

According to \cite[Theorem 5.1]{LS}, $g^{\GC} \cap S$ splits into
$2^{s+t+\delta}$ $S$-classes. Furthermore, $C_{S}(g)$ is a (not necessarily
split) extension of a $2$-group $D$ by $R$, where
$|D| = q^{\dim R_{u}(C_{\GC}(g))}$ and
\begin{equation}\label{cent}
  R = \left(\prod_{i~:~m_{i}\mbox{\tiny { even}}}Sp_{2a_{i}}(q)\right) \times
      \left(\prod_{i~:~m_{i}\mbox{\tiny{ odd}}}I_{2a_{i}}(q)\right)
      \times C_{2}^{t+\delta}.
\end{equation}
Here, $C_2$ denotes a cyclic group of order 2, $I_{2a_{i}}(q) = Sp_{2a_{i}}(q)$ if either $m_{i} = 1$ or there exists
$j$ such that $2k_{j} = m_{i} \pm 1$, and $I_{2a_{i}}(q) = O^{\e_{i}}_{2a_{i}}(q)$
for some $\e_{i} = \pm$ otherwise; in particular, $s$ is the number of
$O$-factors in the above factorization. We will see that these $\e_{i}$ are
related to the type of $g$-invariant quadratic forms discussed in
Lemma \ref{homog}(i).

If $r > 0$, partition $K := \{k_{1}, k_{2}, \ldots, k_{r}\}$ into a
disjoint union
$K_{\delta} \sqcup K_{\delta+1} \sqcup \ldots \sqcup K_{t+\delta}$ of
$t+1$ ``intervals'' of consecutive integers, such that
$$\max\{k \mid k \in K_{i}\} \leq \min\{l \mid l \in K_{i+1}\}-2;$$
this is possible by the definition of the parameter $t$. Now, if
$m_{i} > 1$ is odd and $m_{i} = 2k_{j} \pm 1 = 2k_{j'} \pm 1$ for distinct
$j$ and $j'$, then $|k_{j}-k_{j'}| = 1$ and so $k_{j}$ and $k_{j'}$ must
belong to the same interval $K_{u}$. In this situation, we will say that
$m_{i}$ is {\it linked} to $K_{u}$.
Next, for $\delta \leq u \leq t+\delta$, let
$$V_{u} = \left(\sum_{k_{j} \in K_{u}}V(2k_{j})^{b_{j}}\right)
  \oplus \left( \sum_{m_{i} > 1, \mbox{{\tiny { odd, linked to }}}K_{u}}
  W(m_{i})^{a_{i}}\right).$$
Also, if $I = \{i_{1}, \ldots ,i_{s}\}$, then set
$W_{v} := W(m_{i_{v}})^{a_{i_{v}}}$ for $1 \leq v \leq s$.
Finally, let
$$W_{0} = \sum_{m_{i} = 1 \mbox{\tiny { or }} m_{i} \mbox{\tiny { even}}}
  W(m_{i})^{a_{i}}.$$
Thus we obtain the decomposition
\begin{equation}\label{dec3}
  V|_{\langle g \rangle} = W_{0} \oplus W_{1} \oplus \ldots \oplus W_{s}
  \oplus V_{\delta} \oplus V_{1+\delta} \oplus \ldots \oplus V_{t+\delta}.
\end{equation}

\begin{theorem}\label{uni-Sp}
{Consider the decomposition (\ref{dec3}) for any unipotent element
$g \in S = Sp_{2n}(q)$. Let $\GC = Sp_{2n}(\overline{\F}_{q})$.
Then $g^{\GC} \cap S$ splits into
$2^{s+t+\delta}$ $S$-classes. Each such class is uniquely determined by
the sequence
$\bve = (\al_{1}, \ldots ,\al_{s},\beta_{1}, \ldots ,\beta_{t+\delta})$,
where $\al_{i}, \beta_{j} = \pm$, and every $g$-invariant quadratic form
polarized to $(\cdot,\cdot)$ on $W_{i}$ with $i \geq 1$,
respectively on $V_{j}$ with $j \geq 1$, is
of type $\al_{i}$, respectively $\beta_{j}$. Furthermore, if
the $S$-class of $g$ is determined by $\bve$, then in the factorization
(\ref{cent}) for $R$ the factor $I_{2a_{i}}(q)$ equals
$O^{\al_{v}}_{2a_{i}}(q)$ for $i = i_{v} \in I$, and $Sp_{2a_{i}}(q)$ otherwise.}
\end{theorem}

\begin{proof}
1) First we observe that the invariant $\bve$ is well-defined for $g$.
Indeed, by Lemma \ref{homog}, if $U$ is a $g$-invariant non-degenerate
subspace of $V$ of type $W(m)^{a}$ or $V(m)^{b}$ with $m \geq 3$, then
all $g$-invariant quadratic forms polarized to $(\cdot,\cdot)$ on $U$ have
the same type. This applies in particular to any $W_{i}$ with $i \geq 1$ and
any $V_{j}$ with $j \geq 1$, whence the observation follows. It is also clear
that $\bve$ is the same for all $x \in g^{S}$.

2) Next we aim to show that if two elements
$g, h \in g^{\GC} \cap S$ have the
same invariant $\bve$, then they are conjugate in $S$. Applying
\cite[Lemma 4.2]{LS} and conjugating $h$ by an element in $S$,
we may assume that $g$ and $h$ have the same canonical decomposition
(\ref{dec1}) and the subsequent decomposition (\ref{dec3}).

First we look at the case where the decomposition (\ref{dec3}) reduces
to $V = W_{0}$ or $V = V_{0} \oplus V_{0}$; in particular, $s=t=\delta=0$.
In this case, $g^{\GC} \cap S$ constitutes a single $S$-class by
\cite[Theorem 5.1 (ii)]{LS},
whence $g$ and $h$ are $S$-conjugate.

Next we look at the case where
the decomposition (\ref{dec3}) reduces to $V = V_{u}$. In this case,
$\delta = 1$, $s = t = 0$, and $g^{\GC} \cap S \neq \emptyset$ by
\cite[Theorem 5.1 (i)]{LS}; furthermore, $g^{\GC} \cap S$ splits into two
$S$-classes by \cite[Theorem 5.1 (ii)]{LS}. By
\cite[Theorem 5.1 (i)]{LS}, for each $\e = \pm$ we can pick
$u_{\e} \in g^{\GC} \cap S$ such that some, hence all by Lemma \ref{homog},
$u_{\e}$-invariant quadratic forms polarized to $(\cdot,\cdot)$ on $V$ are
of type $\e$. Also by Lemma \ref{homog} we have
$\bLn(u_{\e}) = \e q^{k}$ for some integer $k$ which depends only on
$g^{\GC}$. It follows that $u_{+}$ and $u_{-}$ are not $S$-conjugate,
whence $g^{\GC} \cap S = (u_{+})^{S} \sqcup (u_{-})^{S}$. Now, denoting
$\beta_{1} = \beta$, we see by Lemma \ref{homog} that
$\bLn(g) = \bLn(h) = \beta q^{k}$ and so $g$ and $h$ cannot be $S$-conjugate
to $u_{-\beta}$. Thus both $g$ and $h$ are $S$-conjugate to $u_{\beta}$.
The same argument also applies to the case where the decomposition
(\ref{dec1}) reduces to $V = W_{v}$.

We have therefore shown that all the pairs of elements
$(g|_{U},h|_{U})$ are conjugate in $Sp(U)$, where $U = W_{0}$ (or
$W_{0} \oplus V_{0}$ if $\delta = 0$), $W_{i}$ with $i > 0$, or $V_{j}$ with
$j > 0$. Since $V$ is the orthogonal sum of those subspaces $U$, we conclude
that $g$ and $h$ are conjugate in $Sp(V) = S$.

3) Consequently, each $g^{S}$ is uniquely determined by the sequence $\bve$.
It remains to determine the factors $I_{2a_{i}}(q)$ of type $O^{\pm}_{2a_{i}}(q)$
in the factorization (\ref{cent}) for $R$. As noted in the proof of
\cite[Theorem 5.1]{LS}, each of the summands in the canonical decomposition
(\ref{dec1}) can be written over $\F_{2}$ and so the Frobenius endomorphism
$F$ stabilizes each of the factors $Sp_{2a_{i}}(q)$, $I_{2a_{i}}(q)$, and
$C_{2}$ which appear in $C_{\GC}(g)/R_{u}(C_{\GC}(u))$ and in $R$. So
without loss of generality we may assume that the decomposition
(\ref{dec3}) reduces to $V = W_{v}$; in particular, $\bve = (\al_{v})$ and
$\dim V = 2a_{i_v}$. We fix a $g$-invariant quadratic form polarized to
$(\cdot,\cdot)$ on $V$ of type $\al_{v}$. Direct computation using relations
(3.7.3) -- (3.7.5) of \cite{W1} shows that the $2'$-parts of $|C_{Sp(V)}(g)|$ and
$|C_{O(V)}(g)|$ are both equal to $|O^{\al_{v}}_{2a_{i_v}}(q)|_{2'}$. On the other hand,
we know that $I_{2a_{i_v}}(q) = O^{\e_{v}}_{2a_{i_v}}(q)$ and so the $2'$-part of
$|C_{Sp(V)}(g)|$ is $|O^{\e_{v}}_{2a_{i_v}}(q)|_{2'}$. It follows that
$\e_{v} = \al_{v}$
and $I_{2a_{i_v}}(q) = O^{\al_{v}}_{2a_{i_v}}(q)$, as stated.
\end{proof}

Define $\HC = O(V \otimes _{\F_{q}} \overline{\F}_{q})$ (notice
that $\HC$ is disconnected) and a
Frobenius endomorphism $F$ on $\HC$ such that
$\HC^{F} = H := O^{\e}_{2n}(q)$ for $\e = \pm$. For any
unipotent element $g \in \HC$ we again consider the canonical decomposition
(\ref{dec1}). Let $I$ be the set of indices $i$ such that $m_{i}$ is odd
and there does {\bf not} exist $j$ such that $2k_{j} = m_{i} \pm 1$, and
let $s := |I|$.
Next, let $t$ be the number of indices $j$ such that $k_{j+1}-k_{j} \geq 2$.
We also fix $\delta \in \{0,1\}$ with $\delta = 1$ precisely when $r > 0$.
If $r > 0$, partition $K := \{k_{1}, k_{2}, \ldots, k_{r}\}$ into a
disjoint union $K_{1} \sqcup K_{2} \sqcup \ldots \sqcup K_{t+\delta}$ of
$t+\delta$ intervals of consecutive integers, such that
$$\max\{k \mid k \in K_{i}\} \leq \min\{l \mid l \in K_{i+1}\}-2.$$
As in the symplectic case, if
$m_{i}$ is odd and $m_{i} = 2k_{j} \pm 1 = 2k_{j'} \pm 1$ for distinct
$j$ and $j'$, then $|k_{j}-k_{j'}| = 1$ and so $k_{j}$ and $k_{j'}$ must
belong to the same interval $K_{u}$. In this situation, we will again say
that $m_{i}$ is {\it linked} to $K_{u}$.
Let $W_{0} = \sum_{2|m_{i}}W(m_{i})^{a_{i}}$ and $I = \{i_{1}, \ldots ,i_{s}\}$.
Next, we define
$$V_{u} = \left(\sum_{k_{j} \in K_{u}}V(2k_{j})^{b_{j}}\right)
  \oplus \left( \sum_{m_{i} \mbox{{\tiny { odd, linked to }}}K_{u}}
  W(m_{i})^{a_{i}}\right)$$
if $r > 0$ and $1 \leq u \leq t+\delta$, and
$W_{v} := W(m_{i_{v}})^{a_{i_{v}}}$ for $1 \leq v \leq s$.
Thus we obtain the decomposition
\begin{equation}\label{dec4}
  V|_{\langle g \rangle} = W_{0} \oplus W_{1} \oplus \ldots \oplus W_{s}
  \oplus V_{1} \oplus V_{2} \oplus \ldots \oplus V_{t+\delta}.
\end{equation}
We say that $g$ is {\it exceptional}
if $V|_{\langle g \rangle} = \sum_{i}W(m_{i})^{a_{i}}$ with
all $m_{i}$ even. Now we exhibit the following
analogue of Theorem \ref{uni-Sp} for orthogonal groups.

\begin{theorem}\label{uni-O}
{Consider the decomposition (\ref{dec3}) for any unipotent element
$g \in \HC = O_{2n}(\overline{\F}_{q})$. Let
$H := O^{\e}_{2n}(q)$ and $K := \Omega^{\e}_{2n}(q)$ for some $\e = \pm$.

{\rm (i)} Then $g^{\HC} \cap H \neq \emptyset$ unless
$H = O^{-}_{2n}(q)$ and $g$ is exceptional,
in which case $g^{\HC} \cap H = \emptyset$.

{\rm (ii)} Assume $g^{\HC} \cap H \neq \emptyset$.
Then $g^{\HC} \cap H$ splits into
$2^{s+t+\delta-1}$ $H$-classes, if $g$ is not exceptional, and constitutes
a single $H$-class, if $g$ is exceptional.

\hspace{4mm}{\rm (a)} Each such $H$-class constitutes a single $K$-class, except when $g$ is exceptional, in which
case it splits into two $K$-classes.

\hspace{4mm}{\rm (b)} Each such $H$-class is uniquely determined by
the sequence
$\bve = (\al_{1}, \ldots ,\al_{s},\beta_{1}, \ldots ,\beta_{t+\delta})$,
where $\al_{i}, \beta_{j} = \pm$, and every $g$-invariant quadratic form
polarized to $(\cdot,\cdot)$ on $W_{i}$ with $i \geq 1$,
respectively on $V_{j}$ with $j \geq 1$, is
of type $\al_{i}$, respectively $\beta_{j}$, and
\begin{equation}\label{type}
  \prod^{s}_{i=1}\al_{i} \cdot \prod^{t+\delta}_{j=1}\beta_{j} = \e.
\end{equation}

\hspace{4mm}{\rm (c)} If the $H$-class of $g \in H$ is determined by $\bve$, then
$C_{H}(g)$ is an extension
of a $2$-group of order $q^{\dim R_{u}(C_{\HC}(g))}$ by $R$, and the factorization
(\ref{cent}) holds for $R$, the factor $I_{2a_{i}}(q)$ equals
$O^{\al_{v}}_{2a_{i}}(q)$ for $i = i_{v} \in I$, and $Sp_{2a_{i}}(q)$ otherwise.
}
\end{theorem}

\begin{proof}
(i) Notice that if $g \in H$ is exceptional
then $g \in K$ as the quasideterminant
$(-1)^{\dim \Ker(g-1)}$ is $1$. Hence the statement follows from
\cite[Theorem 5.1 (i)]{LS}.

(ii) First we consider the case $g$ is exceptional;
in particular, $s = t = \delta = 0$. Then $C_{\HC}(g)$ is
connected by \cite[Theorem 4.20]{LS}, hence $g^{\HC} \cap H$ constitutes
a single $H$-class. The connectedness of $C_{\HC}(g)$ also implies that
$C_{H}(g) \leq H \cap \HC^{\circ} = K$ (the latter equality can be seen by
using quasi-determinant), whence $g^{H}$ splits into two $K$-classes. The
structure of $C_{H}(g)$ is described in \cite[Theorem 5.1 (iii)]{LS}.

From now on we may assume that $g$ is not exceptional. In particular,
$C_{H}(g) \not\leq K$ (see e.g. the proof of \cite[Prop. 4.21]{LS}), hence
$g^{H}$ constitutes a single $K$-class. By the same reason,
$g^{\HC} = g^{\HC^{\circ}}$.

Next, by \cite[Theorem 5.1 (i)]{LS}, every $g$-invariant
quadratic form polarized to $(\cdot,\cdot)$ on $W_{0}$ is of type $+$. Also,
according to Lemma \ref{homog}, all such
forms on $W_{v}$, respectively on $V_{u}$, have the same type $\al_{v}$,
respectively $\beta_{u}$, as long as $m_{i_{v}} > 1$, respectively
$\min \{k \mid k \in K_{u}\} > 1$. Let $\al_{1}$,
respectively $\beta_{1}$, denote the type of such a form on $W_{1}$ with
$m_{i_{1}} = 1$, respectively on $W_{1}$ with $k_{1} = 1$. We claim that in
this situation, $\al_{1}$, respectively $\beta_{1}$, is also uniquely determined
by $g^{H}$, and moreover in all cases (\ref{type}) holds. Indeed, the
latter equality follows from the decomposition (\ref{dec4}) and the fact that
the type of $W_{0}$ is $+$. Assume $k_{1} = 1$. Notice that in this case
$m_{i_{1}} > 1$ (as otherwise it is linked to $K_{1} \ni k_{1} = 1$). Thus all
$\al_{i}$ with $i \geq 1$ and all $\beta_{j}$ with $j \geq 2$ are uniquely determined
by $g^{H}$, and so is $\beta_{1}$ by virtue of (\ref{type}). The same argument
applies to the case $k_{1} > 1$ and $m_{i_{1}} = 1$.

Thus we have shown that the invariant $\bve$ is well defined for $g^{H}$.
Furthermore, Theorem \ref{uni-Sp} implies that, given any $\beta_{u} = \pm$,
there is a $g$-invariant quadratic form polarized to $(\cdot,\cdot)$ on $V_{u}$
of type $\beta_{u}$. (Indeed, the claim is clear if $V_{u}$ does not involve any
summand $W(m_{i})$ with $m_{i} = 1$. The statement also holds in the case
$V_{u}$ involves the summand $W(1)^{2a}$: just write
$V_{u} = Y \oplus W(1)^{2a}$, fix a $g$-invariant quadratic form polarized to
$(\cdot,\cdot)$ on $Y$ say of type $\gam$ and then choose any quadratic form of
type $\gam\beta_{u}$ on $W(1)^{2a}$, on which $g$ acts trivially.) Similarly,
given any $\al_{v} = \pm$, there is a $g$-invariant quadratic form polarized to
$(\cdot,\cdot)$ on $W_{v}$ of type $\al_{v}$. Thus there are exactly
$2^{s+t+\delta-1}$ possible values for the sequence $\bve$ subject to the condition
(\ref{type}), whence the conclusion (b) follows.

Finally, the arguments given in part 3 of the proof of Theorem \ref{uni-Sp}
also show that $I_{2a_{v}}(q) = O^{\al_{v}}_{2a_{i_v}}(q)$.
\end{proof}

\subsection{The induced virtual character
$\bLn = \Ind^{GL_{2n}(q)}_{Sp_{2n}(q)}(\bln)$}
Denote $G = GL_{2n}(q)$, $S = Sp_{2n}(q)$, and $\GC = Sp_{2n}(\overline{\F}_{q})$ as usual. For any unipotent element
$g \in G$, if $g^{G} \cap S = \sqcup^{N}_{i=1}(g_{i})^{S}$ splits into $N$
$S$-classes, then the definition of induced characters yields
\begin{equation}\label{sum}
  \bLn(g) = |C_{G}(g)| \cdot \sum^{N}_{i=1}\frac{\bln(g_{i})}{|C_{S}(g_{i})|}.
\end{equation}
Clearly, $g^{G} \cap S$ is just the set of all unipotent elements $x \in S$ with
the same Jordan canonical form as of $g$. It follows that $g^{G} \cap S$ is the
union of $g^{\GC} \cap S$ for all $\GC$-classes $g^{\GC}$ with the same
Jordan canonical form.

Consider the canonical decomposition (\ref{dec1}) for $g^{\GC} \cap S$. By
Lemmas \ref{mult2} and \ref{homog}, $\bln(g) = 0$ if $k_{1} = 1$. Assume that
$r > 0$ and $k_{1} > 1$; in particular, $\delta = 1$. By Theorem \ref{uni-Sp},
$g^{\GC} \cap S$ splits into $2^{s+t+1}$ $S$-classes which are uniquely determined
by the sequence $\bve$. Now we use the decomposition (\ref{dec3}) and Lemma
\ref{mult2} to compute $\bln(g)$:
$$\bln(g) = \prod^{s}_{i=0}\bl_{(\dim W_i)/2}(g|_{W_i}) \cdot
  \prod^{t+1}_{j=1}\bl_{(\dim V_j)/2}(g|_{V_j}).$$
First we look at $g|_{W_0}$. Allowing $a_{1}$ to be zero if necessary, we may assume
that $m_{1} = 1$. As mentioned in the proof of Theorem \ref{uni-O}, all
$g$-invariant quadratic forms polarized to $(\cdot,\cdot)$ on $W_0$ are of type
$+$. Applying Lemma \ref{homog}, we see that
$$\bl_{(\dim W_0)/2}(g|_{W_0}) = q^{a_{1}+2\sum_{2|m_{i}}a_{i}}.$$
On the other hand, by Lemmas \ref{mult2} and \ref{homog} we have
$$\bl_{(\dim W_v)/2}(g|_{W_v}) = \al_{v}q^{2a_{i_{v}}}$$
for $1 \leq v \leq s$ , and
$$\bl_{(\dim V_u)/2}(g|_{V_u}) = \beta_{u}q^{\sum_{k_{j} \in K_{u}}b_{j}
  + \sum_{m_{i} > 1, \mbox{{\tiny { odd, linked to }}}K_{u}}2a_{i}}$$
for $1 \leq u \leq t+1$. Thus, there is an explicit constant $C$ depending only
on $g^{\GC} \cap S$ such that $\bln(g) = [\bve] \cdot q^{C}$, if
the conjugacy class $g^{S}$ is determined by $\bve$ and
$[\bve] := \prod^{s}_{v=1}\al_{v} \cdot \prod^{t+1}_{u=1}\beta_{u}$. Recall we are
assuming that $r > 0$ and $k_{1} > 1$. Then we can pair up the $2^{s+t+1}$
$S$-classes in $g^{\GC} \cap S$ into $2^{s+t}$ pairs, each consisting of
$(h_{+})^{S}$ and $(h_{-})^{S}$, determined by
$\bve_{+}$ and $\bve_{-}$, which differ only at $\beta_{1} = \pm$ and have
the same $\al_{i}$ with $i > 0$ and the same $\beta_{j}$ with $j > 1$. The
above computation shows that $\bln(h_{-}) = -\bln(h_{+})$. On the other hand,
$|C_{S}(h_{+})| = |C_{S}(h_{-})|$ by Theorem \ref{uni-Sp}. Hence the contributions
of the pair $(h_{+})^{S}$ and $(h_{-})^{S}$ to $\bLn(g)$ in (\ref{sum}) cancel out
each other, and so the total contribution of $g^{\GC} \cap S$ in
(\ref{sum}) is $0$.

We have shown that the only nonzero contributions in (\ref{sum}) can only
come from the classes in $g^{\GC} \cap S$ with $r = t = \delta = 0$. In particular,
$\bLn(g) = 0$ if the multiplicity $c_{i}$ of some Jordan block $J_{i}$ in the
Jordan canonical form $\sum^{\infty}_{i=1}c_{i}J_{i}$ of $g$ is odd. Thus we may now
assume that $c_{i}$ is even for all $i$, and the decompositions
(\ref{dec1}) and (\ref{dec3}) of $g$ reduce to
$$V = \sum_{i}W(m_{i})^{a_{i}} = W_{0} \oplus W_{1} \oplus \ldots \oplus W_{s}$$
(so $a_{i} = c_{m_{i}}$). We will assume that the $S$-class of
$g_{\bve} \in S$ is determined by $\bve = (\al_{1}, \ldots ,\al_{s})$.
The above computation then shows that
$$\bln(g_{\bve}) = [\bve] \cdot q^{c_{1}/2+\sum_{i > 1}c_{i}}$$
with $[\bve] = \prod^{s}_{j=1}\al_{j}$, and, according to Theorem \ref{uni-Sp},
$$|C_{S}(g_{\bve})| = q^{D} \cdot \prod_{m_{i} = 1 \mbox{{\tiny { or }}} 2|m_{i}}
  |Sp_{2a_{i}}(q)| \cdot \prod^{s}_{v=1}|O^{\al_{v}}_{2a_{i_{v}}}(q)|$$
where $I = \{i_{1}, \ldots ,i_{s}\}$ is the set of indices $i$ such that
$m_{i} > 1$ is odd, as before.

Observe that
$$\sum_{\al = \pm}\frac{\al 1}{|O^{\al}_{2a}(q)|} = \frac{q^{a}}{|Sp_{2a}(q)|}.$$
Hence, the $2^{s}$ contributions of all $g_{\bve}$ in (\ref{sum}) sum up to
$$\bLn(g) = \frac{|C_{G}(g)| \cdot q^{-D + c_{1}/2+\sum_{i > 1}c_{i}}}
  {\prod_{m_{i} = 1 \mbox{{\tiny { or }}} 2|m_{i}}|Sp_{2a_{i}}(q)|} \cdot
  \prod^{s}_{v=1}\frac{q^{a_{i_{v}}}}{|Sp_{2a_{i_{v}}}(q)|} = $$
$$=  \frac{|C_{G}(g)| \cdot q^{-D + \sum_{i > 1}c_{i}+\sum_{i \mbox{{\tiny { odd}}}}c_{i}/2}}
  {\prod_{i}|Sp_{c_{i}}(q)|},$$
where we use the convention that $|Sp_{0}(q)| = 1$. By \cite{W1},
$$D = \dim R_{u}(C_{\GC}(g)) = \sum_{i < j}ic_{i}c_{j} + \sum_{i}(i-1)c_{i}^{2}/2
    + \sum_{2|i}c_{i}/2 + \sum_{i > 1 \mbox{{\tiny { odd}}}}c_{i}.$$
Putting everything together, we obtain

\begin{theorem}\label{induced}
{Let $g \in Sp_{2n}(q)$ be a unipotent element with Jordan canonical form
$\sum^{\infty}_{i=1}c_{i}J_{i}$. Then
$$\bLn(g) = q^{\frac{1}{2}\sum_{i}c_{i}-\sum_{i < j}ic_{i}c_{j} - \frac{1}{2}\sum_{i}(i-1)c_{i}^{2}}
  \cdot \frac{|C_{GL_{2n}(q)}(g)|}{\prod_{i}|Sp_{c_{i}}(q)|}$$
if all $c_{i}$ are even, and $\bLn(g) = 0$ otherwise.
\hfill $\Box$}
\end{theorem}

Noting from Theorem \ref{induced} that $\bLn(g) \geq 0$, the following corollary is immediate.

\begin{cor} A random unipotent element of $Sp_{2n}(q)$ with given Jordan canonical form fixes a positive
type quadratic form with probability at least as large as that of fixing a negative type quadratic form.
\end{cor}

\subsection{Main result}
To state the main result of this section requires notation about partitions, much of it standard \cite{Mac}. Let $\lambda$ be a partition of some nonnegative integer $|\lambda|$ into parts $\lambda_1 \geq \lambda_2 \geq \cdots$.
The symbol $m_i(\lambda)$ will denote the number of parts of $\lambda$ of size $i$, and $\lambda'$ is the partition dual to $\lambda$ in the sense that $\lambda_i'=m_i(\lambda)+m_{i+1}(\lambda)+\cdots$. Let $n(\lambda)=\sum_i {\lambda_i' \choose 2}$. Let $l(\lambda)$ denote the number of parts of $\lambda$ and $o(\lambda)$ the number of odd parts of $\lambda$.

 It is often helpful to view partitions diagrammatically.
 The diagram associated to $\lambda$ is the set of ordered pairs $(i,j)$ of
 integers such that $1 \leq j \leq \lambda_i$.
 We use the convention that the row index $i$ increases as one
 goes downward and the column index $j$ increases as one goes across.
 So the diagram of the partition $(5,4,4,1)$ is
 \[ \begin{array}{c c c c c}
             \framebox{}& \framebox{}& \framebox{}& \framebox{} & \framebox{}  \
\\
     \framebox{}& \framebox{} & \framebox{} & \framebox{} &\\
     \framebox{}& \framebox{} & \framebox{} & \framebox{}& \\
             \framebox{} &&&&
       \end{array} \]
\newline
and one has that $n(\lambda)=15, l(\lambda)=4$, and $o(\lambda)=2$.

\begin{theorem} \label{con} Let $p^{\pm}(\lambda)$ denote the proportion of elements of $O^{\pm}_{2n}(q)$ which are unipotent and have
$GL_{2n}(q)$ rational canonical form of type $\lambda$.
\begin{enumerate}
\item $p^+(\lambda) + p^-(\lambda)$ is equal to $0$ unless $|\lambda|=2n$ and all odd parts of $\lambda$ have even multiplicity. If $|\lambda|=2n$ and all odd parts of $\lambda$ have even multiplicity, then \begin{eqnarray*} & & p^+(\lambda) + p^-(\lambda)\\ & = & \frac{q^{l(\lambda)}}{q^{n(\lambda)+\frac{|\lambda|}{2}+\frac{o(\lambda)}{2}} \prod_i (1-1/q^2) (1-1/q^4) \cdots (1-1/q^{2 \lfloor m_i(\lambda)/2 \rfloor})}.\end{eqnarray*}

\item $p^+(\lambda) - p^-(\lambda)$ is equal to $0$ unless $|\lambda|=2n$ and all parts of $\lambda$ have even multiplicity. If $|\lambda|=2n$ and all parts of $\lambda$ have even multiplicity, then \[ p^+(\lambda) - p^-(\lambda) = \frac{1}{q^{\sum (\lambda_i')^2/2} \prod_{i \geq 1} (1-1/q^2)(1-1/q^4) \cdots (1-1/q^{m_i(\lambda)})}.\]
\end{enumerate}
\end{theorem}

\begin{proof} Let $g$ be a unipotent element of $GL_{2n}(q)$ of type
$\lambda$, and let $d(g)$ be the dimension of the kernel of $g-1$.
From the above discussion and denoting the principal character by
$[1]$, it follows that
\begin{eqnarray*}  & & \Ind_{O^{+}_{2n}(q)}^{Sp_{2n}(q)}[1](g) +
\Ind_{O^{-}_{2n}(q)}^{Sp_{2n}(q)}[1](g)\\
 & = & 1 + \ran(g) + \rbn(g) + 1 + 2 \left[ \sum^{(q-2)/2}_{i=1}\taui \right] (g) \\
& = & \frac{q^{d(g)}-1}{q-1} + 1 + 2 \left[ \sum^{(q-2)/2}_{i=1}\taui \right] (g) \\
& = & \frac{q^{d(g)}-1}{q-1} + 1 + (q-2) \frac{q^{d(g)}-1}{q-1} \\
& = & q^{d(g)}. \end{eqnarray*} The first equality used formula \eqref{for-p} in Subsection \ref{permchar}, the second equality used the fact that $[1]+\ran + \rbn$ is the permutation character of $Sp_{2n}(q)$ on lines, and the third equality used formula \eqref{val} in Subsection \ref{permchar}.

Let $C$ denote the $GL_{2n}(q)$ class of $g$ and let $D_i$ denote the $Sp_{2n}(q)$ conjugacy classes into which $C \cap Sp_{2n}(q)$ splits. Using the fact that $q^{d(\cdot)}$ is constant on conjugacy classes of $GL_{2n}(q)$, it follows by the general formula for induced characters (page 34 of \cite{FH}) that
\[ \Ind_{Sp_{2n}(q)}^{GL_{2n}(q)}[q^{d(\cdot)}](g)  =  q^{d(g)} |C_{GL_{2n}(q)}(g)| \frac{|C \cap Sp_{2n}(q)|}{|Sp_{2n}(q)|},\] where $C_{GL_{2n}(q)}(g)$ denotes the centralizer of $g$ in $GL_{2n}(q)$. A formula for $\frac{|C \cap Sp_{2n}(q)|}{|Sp_{2n}(q)|}$ in even characteristic appears in Theorem 5.2 of \cite{FG1}; using this and transitivity of induction, one concludes that
\begin{eqnarray*} \label{arr}
& & \Ind_{O^{+}_{2n}(q)}^{GL_{2n}(q)}[1](g) + \Ind_{O^{-}_{2n}(q)}^{GL_{2n}(q)}[1](g)\\
& = & \Ind_{Sp_{2n}(q)}^{GL_{2n}(q)}[q^{d(\cdot)}](g)\\
& = & q^{d(g)} |C_{GL_{2n}(q)}(g)|\frac{|C \cap Sp_{2n}(q)|}{|Sp_{2n}(q)|} \\
& = & \frac{|C_{GL_{2n}(q)}(g)| \cdot q^{l(\lambda)}}{q^{n(\lambda)+n+o(\lambda)/2} \prod_i (1-1/q^2) (1-1/q^4) \cdots (1-1/q^{2 \lfloor m_i(\lambda)/2 \rfloor})}.
\end{eqnarray*} Again using the general formula for induced characters, one has that
\[ p^+(\lambda) + p^-(\lambda) = \frac{1}{|C_{GL_{2n}(q)}(g)|} \left[ \Ind_{O^+_{2n}(q)}^{GL_{2n}(q)}[1](g) + \Ind_{O^-_{2n}(q)}^{GL_{2n}(q)}[1](g) \right], \] and part 1 follows.

By the general formula for induced characters and Theorem \ref{induced}, one concludes that
\begin{eqnarray*}
& & p^+(\lambda) - p^-(\lambda)\\
 & = & \frac{1}{|C_{GL_{2n}(q)}(g)|} \left[ \Ind_{O^+_{2n}(q)}^{GL_{2n}(q)}[1](g) - \Ind_{O^-_{2n}(q)}^{GL_{2n}(q)}[1](g) \right] \\
& = & \frac{\bLn(g)}{|C_{GL_{2n}(q)}(g)|} \\
& = & \frac{q^{\frac{1}{2} \sum_{i \geq 1} m_i(\lambda) - \sum_{i<j} i m_i(\lambda) m_j(\lambda) - \frac{1}{2} \sum_i (i-1) m_i(\lambda)^2}}{\prod_i |Sp_{m_i(\lambda)}(q)|} \\
& = & \frac{1}{q^{\frac{1}{2} \sum_i i m_i(\lambda)^2 + \sum_{i<j} im_i(\lambda)m_j(\lambda)} \prod_{i \geq 1} (1-1/q^2)(1-1/q^4) \cdots (1-1/q^{m_i(\lambda)})}.
\end{eqnarray*} Since $\lambda_j'= m_j(\lambda)+m_{j+1}(\lambda)+ \cdots$, one checks that  \[ \sum_j (\lambda_j')^2 = \sum_j [jm_j(\lambda) + 2 \sum_{i<j} im_i(\lambda)] m_j(\lambda),\] which completes the proof.
\end{proof}

{\it Remark:} Comparing the expression in part 2 with the formula for centralizer sizes in $GL_n(q^2)$ (page 181 of \cite{Mac}), one
concludes that when all $m_i(\lambda)$ are even, $p^+(\lambda)-p^-(\lambda)$ is equal to the proportion of elements in $GL_n(q^2)$ which are unipotent and have part $i$ occur with multiplicity $m_i(\lambda)/2$.

\section{Orthogonal groups in odd dimensions} \label{odddim}
Let $q$ be a power of $2$ as above. It is well known that
$O_{2n+1}(q) \simeq Sp_{2n}(q)$, and this isomorphism can be realized as
follows. Let $U = \F_{q}^{2n+1}$ be endowed with a non-degenerate quadratic
form $Q$. Since $\dim U$ is odd, the associated symplectic form
$(\cdot,\cdot)$ has a $1$-dimensional radical
$J = \langle v \rangle_{\F_{q}}$, with $Q(v) \neq 0$
and $g(v) = v$ for all $g \in O(Q)$. Then $O(Q)$ preserves
the non-degenerate symplectic form on $W := U/J = \F_{q}^{2n}$ induced by
$(\cdot,\cdot)$, and this action induces the isomorphism
$O_{2n+1}(q) = O(Q) \simeq Sp(W) = Sp_{2n}(q)$. Even though
the $O(Q)$-module $U$ is indecomposable, we will show that it is
easy to relate the Jordan canonical form of any element
$g \in O(Q)$ in $GL(U)$ and $GL(W)$. Let $J_{k}(\al)$ denote the
$k \times k$ Jordan block with eigenvalue $\al \in \overline{\F}_{q}$.

\begin{lemma}\label{sp-O}
{Keep the above notation and let $g \in O(Q)$. Then the Jordan canonical
form for $g$ in $GL(U \otimes \overline{\F}_{q})$ is just the direct sum
of the Jordan canonical form for $g$ in
$GL(W \otimes \overline{\F}_{q})$ and the Jordan block $J_{1}(1)$.}
\end{lemma}

\begin{proof}
To simplify the notation, we will extend the scalars to $\overline{\F}_{q}$
and denote the corresponding spaces also by $U$, $J$, and $W$. Let
$\sum_{i,\al}a_{i}(\al)J_{i}(\al)$, respectively
$\sum_{i,\al}b_{i}(\al)J_{i}(\al)$, be the Jordan canonical form for
$g$ in $GL(U)$, respectively in $GL(W)$. We need to show that
$a_{i}(\al) - b_{i}(\al)$ equals $1$ if $(i,\al) = (1,1)$ and $0$ otherwise.
For any $j \geq 0$ and $\al \in \overline{\F}_{q}$, let
$$U_{j}(\al) = \{ x \in U \mid (g- \alpha \cdot 1_U)^{j}(x) = 0\}$$
and similarly for
$W_{j}(\al)$; also set $U_{0}(\al) = 0$ and $W_{0}(\al) = 0$. Then it is easy
to check for $j \geq 1$ that
$\dim U_{j}(\al) = \sum_{i} \min\{i,j\} \cdot a_{i}(\al)$. It follows that
$\dim U_{j+1}(\al)/U_{j}(\al) = \sum_{i \geq j+1}a_{i}(\al)$, and so
$$a_{j}(\al) = -(\dim U_{j+1}(\al)) + 2(\dim U_{j}(\al)) - (\dim U_{j-1}(\al)),
$$
and similarly
$$b_{j}(\al) = -(\dim W_{j+1}(\al)) + 2(\dim W_{j}(\al)) - (\dim W_{j-1}(\al)).
$$
Since $g$ acts trivially on $J$, it is straightforward to check that
$\dim U_{j}(\al) = \dim W_{j}(\al)$ for $\al \neq 1$, whence
$a_{j}(\al) = b_{j}(\al)$ for any such $\al$. Let $j \geq 1$ and
$u+J \in W_{j}(1)$. Denoting $w := (g-1)^{j-1}(u)$, we have
$w+J \in  W_{1}(1)$, i.e. $g(w) = w +av$ for some $a \in \overline{\F}_{q}$.
But then $Q(w) = Q(w+av) = Q(w)+a^{2}Q(v)$ (since $(v,U) = 0$), and so
$a = 0$ as $Q(v) \neq 0$. Thus $(g-1)^{j}(u) = (g-1)w = 0$, i.e.
$u \in U_{j}(1)$. We conclude that $W_{j}(1) = U_{j}(1)/J$, and so
$\dim U_{j}(1) = 1+\dim W_{j}(1)$ for $j \geq 1$. Recall that
$U_{0}(1) = W_{0}(1) = 0$. Hence $a_{i}(1) - b_{i}(1)$ equals $1$ if
$i = 1$ and $0$ otherwise.
\end{proof}

Cycle indices for even characteristic symplectic groups were derived in \cite{FG1},
and Lemma \ref{sp-O} reduces the cycle index of $O^{\pm}(2n+1,q)$ to that of $Sp(2n,q)$. Thus we can
focus attention on $O^{\pm}(2n,q)$, which we do in the next section.

\section{Cycle index} \label{cycleindex}
Recall that we are interested in the $GL_{2n}(q)$ rational canonical form of random elements of $O^{\pm}_{2n}(q)$.
The rational canonical forms of $GL_{2n}(q)$ are parameterized by associating a partition $\lambda_{\phi}$ to each monic, non-constant irreducible polynomial $\phi$ over the finite field $\mathbb{F}_q$, such that
\begin{enumerate}
\item $|\lambda_z|=0$
\item $\sum_{\phi} |\lambda_{\phi}| \cdot \deg(\phi)=2n$.
\end{enumerate} Here $\deg(\phi)$ denotes the degree of $\phi$, and $|\lambda_{\phi}|$ is the size of $\lambda_{\phi}$.
For elements in $O^{\pm}_{2n}(q)$, it follows from Wall \cite{W1} that there are additional restrictions:
\begin{enumerate}
\item $\lambda_{\phi}=\lambda_{\phi^*}$, where $\phi^*(z)=\phi(0)^{-1}z^n \phi(z^{-1})$.
\item The odd parts of $\lambda_{z-1}$ occur with even multiplicity.
\end{enumerate}

For the remainder of this section, we use the notation:
\[ A(\phi,\lambda_{\phi},i) =
\left\{
\begin{array}{ll} |U_{m_i(\lambda_{\phi})}(q^{\deg(\phi)/2})| & \ if \ \phi=\phi^* \\
|GL_{m_i(\lambda_{\phi})}(q^{\deg(\phi)})|^{1/2} & \ if \ \phi \neq \phi^*
\end{array} \right.\]
We remind the reader that $|GL_{n}(q)|=q^{n^2}(1-1/q)\cdots (1-1/q^n)$ and that the size of $U_{n}(q)$ is $(-1)^n |GL_n(-q)|$. For $\phi \neq z-1$, we define
$B(\phi,\lambda_{\phi})$ as
\[
q^{\deg(\phi)[\sum_{h<i} hm_h(\lambda_{\phi})m_i(\lambda_{\phi})+\frac{1}{2}
\sum_i (i-1)m_i(\lambda_{\phi})^2]} \prod_i A(\phi,\lambda_{\phi},i). \]

Next we give an explicit formula for the cycle index of the orthogonal groups in even characteristic.
We let $x_{\phi,\lambda}$ be variables, and $\lfloor y \rfloor$ denote the largest integer not exceeding
$y$.

\begin{theorem} \label{cycindex}
The following statements hold.

\begin{enumerate}
\item
\begin{eqnarray*} & & 1+\sum_{n \geq 1}
\frac{u^{2n}}{|O^+_{2n}(q)|} \sum_{g \in O^+_{2n}(q)} \prod_{\phi}
x_{\phi,\lambda_{\phi}(g)} + \sum_{n \geq 1}
\frac{u^{2n}}{|O^-_{2n}(q)|} \sum_{g \in O^-_{2n}(q)} \prod_{\phi}
x_{\phi,\lambda_{\phi}(g)} \\
& = & \left( \sum_{|\lambda| \
even \atop i \ odd \Rightarrow m_i \ even} \frac{ x_{z-1,\lambda}
u^{|\lambda|} q^{l(\lambda)} }{q^{n(\lambda)+\frac{|\lambda|}{2}+\frac{o(\lambda)}{2}}
\prod_i (1-1/q^2) \cdots (1-1/q^{2 \lfloor \frac{m_i(\lambda)}{2}
\rfloor})} \right)\\
& & \cdot \prod_{\phi = {\phi^*} \atop \phi
\neq z-1} \left( \sum_{\lambda} \frac{x_{\phi,\lambda}
u^{|\lambda| \cdot \deg(\phi)}}{B(\phi,\lambda)} \right)
\prod_{\{\phi,\phi^*\} \atop \phi \neq \phi^*} \left( \sum_{\lambda}
\frac{x_{\phi,\lambda}x_{\phi^*,\lambda} u^{2|\lambda| \cdot
\deg(\phi)}}{B(\phi,\lambda) B(\phi^*,\lambda)} \right)
\end{eqnarray*}

\item
\begin{eqnarray*} & & 1+\sum_{n \geq 1}
\frac{u^{2n}}{|O^+_{2n}(q)|} \sum_{g \in O^+_{2n}(q)} \prod_{\phi}
x_{\phi,\lambda_{\phi}(g)} - \sum_{n \geq 1}
\frac{u^{2n}}{|O^-_{2n}(q)|} \sum_{g \in O^-_{2n}(q)} \prod_{\phi}
x_{\phi,\lambda_{\phi}(g)}\\
& = & \left( \sum_{\lambda \atop all \ m_i(\lambda) \ even}
\frac{ x_{z-1,\lambda} u^{|\lambda|}}{q^{\sum (\lambda_i')^2/2} \prod_{i \geq 1} (1-1/q^2)(1-1/q^4) \cdots (1-1/q^{m_i(\lambda)})          } \right)\\
& & \cdot \prod_{\phi =
\phi^* \atop \phi \neq z - 1} \left( \sum_{\lambda}
\frac{x_{\phi,\lambda} (-1)^{|\lambda|} u^{|\lambda| \cdot \deg(\phi)}}{B(\phi,\lambda)} \right)
\prod_{\{\phi,\phi^*\} \atop \phi \neq \phi^*} \left( \sum_{\lambda}
\frac{x_{\phi,\lambda}x_{\phi^*,\lambda} u^{2|\lambda| \cdot \deg(\phi)}}{B(\phi,\lambda)
B(\phi^*,\lambda)} \right)
\end{eqnarray*}
\end{enumerate}
\end{theorem}

\begin{proof} Consider the first part. The coefficient of $u^{2n} \prod_{\phi} x_{\phi,\lambda_{\phi}}$ on the left-hand side is the sum of the proportions of elements in $O^{\pm}_{2n}(q)$ with rational canonical form data $\{\lambda_{\phi}\}$ in $GL_{2n}(q)$. By Theorem 3.7.4 of \cite{W1} and part 1 of Theorem \ref{con}, this is equal to the coefficient of $u^{2n} \prod_{\phi} x_{\phi,\lambda_{\phi}}$ on the right-hand side, yielding the first assertion. The second assertion is proved similarly, using part 2 of Theorem \ref{con}.
\end{proof}

{\it Remark}: $\Omega^{\pm}_{2n}(q)$ is defined as the index 2 subgroup of $O^{\pm}_{2n}(q)$ with the property that $l(\lambda_{z-1}(g))$ is even. Thus our techniques can be easily modified to study these groups as well. Namely in the left-hand side of Theorem \ref{cycindex}, one replaces the sum over $g \in O^{\pm}_{2n}(g)$ by the sum over $g \in \Omega^{\pm}_{2n}(q)$ (but leaving the $|O^{\pm}_{2n}(q)|$ unchanged), and in the right hand terms adds the additional restriction that the number of parts of $\lambda_{z-1}$ is even.

\section{Enumerative applications} \label{enumerate}

In this section we present a small sample of enumerative applications of the cycle indices from Section \ref{cycleindex}.

\subsection{Example 1: Enumeration of elements by dimension of fixed space}

Let $p_{2n}^{\pm}(k)$ denote the probability that an element of $O^{\pm}_{2n}(q)$ has a fixed space of dimension $k$.
These probabilities were computed in a long paper of Rudvalis and Shinoda \cite{RS} using Moebius inversion. Theorem \ref{fixdim} shows how to compute these probabilities using cycle indices.

\begin{theorem} \label{fixdim} (\cite{RS})
The following statements hold.
\begin{enumerate}
\item  \begin{eqnarray*} p_{2n}^{\pm}(2k) & = & \frac{q^k}{2|GL_k(q^2)|} \sum_{j=0}^{n-k} \frac{(-1)^j}{q^{(2k-1)j} (q^{2j}-1) \cdots (q^4-1)(q^2-1)}\\ & & \pm \frac{1}{2} \frac{(-1)^{n-k}}{q^{2k(n-k)} |GL_k(q^2)| (q^{2(n-k)}-1) \cdots (q^4-1)(q^2-1)}. \end{eqnarray*}

\item $p_{2n}^{\pm}(2k+1) = \frac{1}{2 q^k |GL_k(q^2)|} \sum_{j=0}^{n-k-1} \frac{(-1)^j}{q^{j^2+2(k+1)j} (1-1/q^2)(1-1/q^4) \cdots (1-1/q^{2j})}$.
\end{enumerate}
\end{theorem}

\begin{proof} By part 1 of Theorem \ref{cycindex}, $p_{2n}^+(2k)+p_{2n}^-(2k)$ is the coefficient of $u^{2n}$ in
\begin{eqnarray*}
& &  \sum_{l(\lambda)=2k \atop i \ odd \Rightarrow m_i \ even} \frac{
u^{|\lambda|} q^{2k}}{q^{n(\lambda)+\frac{|\lambda|}{2}+\frac{o(\lambda)}{2}}
\prod_i (1-1/q^2) (1-1/q^4) \cdots (1-1/q^{2 \lfloor \frac{m_i(\lambda)}{2}
\rfloor})} \\
& & \cdot \prod_{\phi = {\phi^*} \atop \phi
\neq z-1} \left( \sum_{\lambda} \frac{
u^{|\lambda| \cdot \deg(\phi)}}{B(\phi,\lambda)} \right)
\prod_{\{\phi,\phi^*\} \atop \phi \neq \phi^*} \left( \sum_{\lambda}
\frac{u^{2|\lambda| \cdot
\deg(\phi)}}{B(\phi,\lambda) B(\phi^*,\lambda)} \right) \\
& = & \frac{u^{2k} q^k}{q^{2k^2}(1-u^2/q)(1-1/q^2) \cdots (1-u^2/q^{2k-1})(1-1/q^{2k})} \\
& & \cdot \prod_{\phi = {\phi^*} \atop \phi
\neq z-1} \left( \sum_{\lambda} \frac{
u^{|\lambda| \cdot \deg(\phi)}}{B(\phi,\lambda)} \right)
\prod_{\{\phi,\phi^*\} \atop \phi \neq \phi^*} \left( \sum_{\lambda}
\frac{u^{2|\lambda| \cdot
\deg(\phi)}}{B(\phi,\lambda) B(\phi^*,\lambda)} \right) \\
& = & \frac{u^{2k} q^k}{q^{2k^2}(1-u^2/q)(1-1/q^2) \cdots (1-u^2/q^{2k-1})(1-1/q^{2k})} \\
& & \cdot \prod_{\phi = {\phi^*} \atop \phi
\neq z-1} \prod_{i \geq 1} \left(1+(-1)^i(u^2/q^{i})^{\deg(\phi)/2} \right)^{-1}
\prod_{\{\phi,\phi^*\} \atop \phi \neq \phi^*} \prod_{i \geq 1} \left( 1-(u^2/q^{i})^{\deg(\phi)} \right)^{-1}. \end{eqnarray*}
The first equality used Theorem 3 of \cite{F2}, together with the fact from \cite{FG1}
that the $\lambda_{z-1}$ part of an element of $Sp_{2n}(q)$ has the same behavior in odd and
even characteristic. The second equality used an identity from \cite{F1}. Using
Lemma 1.3.17 (parts a and d) of \cite{FNP}, this becomes \begin{eqnarray*}
& & \frac{u^{2k} q^k}{q^{2k^2}(1-u^2/q)(1-1/q^2) \cdots (1-u^2/q^{2k-1})(1-1/q^{2k})} \\
& & \cdot \left[ \frac{\prod_{i \geq 1} (1-u^2/q^{2i-1})}{1-u^2} \right] \\
& = & \frac{u^{2k} q^k \prod_{i \geq k+1} (1-u^2/q^{2i-1})}{(1-u^2) |GL_k(q^2)|}. \end{eqnarray*}

It now follows from an identity of Euler
(\cite{A}, p. 19) that \begin{equation} \label{2} p_{2n}^+(2k)+p_{2n}^-(2k) = \frac{q^k}{|GL_k(q^2)|} \sum_{j=0}^{n-k} \frac{(-1)^j}{q^{(2k-1)j} (q^{2j}-1) \cdots (q^2-1)}.\end{equation}

Part 2 of Theorem \ref{cycindex} gives that $p_{2n}^+(2k)-p_{2n}^-(2k)$ is the coefficient of $u^{2n}$ in
\[ \prod_{i \geq 1} (1-u^2/q^{2i}) \cdot \sum_{l(\lambda)=2k \atop all \ m_i(\lambda) \ even} \frac{u^{|\lambda|}}{q^{\sum (\lambda_i')^2/2} \prod_{i \geq 1} (1-1/q^2) \cdots (1-1/q^{m_i(\mu)})} .\] Here the term $\prod_{i \geq 1} (1-u^2/q^{2i})$ comes from the polynomials other than $z-1$
by an argument similar to that in the previous paragraph. By the remark after Theorem \ref{con} and Theorem 5 of \cite{F3}, this is the coefficient of $u^n$ in \[ \frac{u^k}{|GL_k(q^2)|} \frac{\prod_{i \geq 1} (1-u/q^{2i})}{\prod_{i=1}^k (1-u/q^{2i})} = \frac{u^k}{|GL_k(q^2)|} \prod_{i \geq k+1} (1-u/q^{2i}).\] It now follows from an identity of Euler (\cite{A}, p. 19) that \begin{equation} \label{3} p_{2n}^+(2k)-p_{2n}^-(2k) =
\frac{(-1)^{n-k}}{q^{2k(n-k)} |GL_k(q^2)| (q^{2(n-k)}-1) \cdots (q^4-1)(q^2-1)}.\end{equation} Combining equations \eqref{2} and \eqref{3} proves part 1 of the theorem.

For part 2 of the theorem, arguing as in part 1 of the theorem gives that $p_{2n}^+(2k+1)+p_{2n}^-(2k+1)$ is the coefficient of $u^{2n}$ in \begin{eqnarray*} & & \frac{u^{2k+2}}{q^{2k^2+k}(1-u^2/q)(1-1/q^2) \cdots (1-1/q^{2k})(1-u^2/q^{2k+1})}\\
 & & \cdot \frac{\prod_{i \geq 1} (1-u^2/q^{2i-1})}{1-u^2} \\
& = & \frac{u^{2k+2}}{q^{2k^2+k}(1-1/q^2) (1-1/q^4) \cdots (1-1/q^{2k})} \frac{\prod_{i \geq k+1} (1-u^2/q^{2i+1})}{1-u^2}.
\end{eqnarray*} Again using Euler's identity, this is equal to \[ \frac{1}{q^k |GL_k(q^2)|} \sum_{j=0}^{n-k-1} \frac{(-1)^j}{q^{j^2+2(k+1)j} (1-1/q^2)(1-1/q^4) \cdots (1-1/q^{2j})}.\] Part 2 of Theorem \ref{cycindex} implies that $p_{2n}^+(2k+1)-p_{2n}^-(2k+1)=0$ (if all parts of $\lambda_{z-1}$ occur with even multiplicity, the total number of parts can't be odd), and the result follows.
\end{proof}

Since $\Omega^{\pm}(2n,q)$ is the index 2 subgroup of $O^{\pm}(2n,q)$ consisting of elements with an even dimensional fixed space, the following corollary of Theorem \ref{fixdim} is immediate.

\begin{cor} \label{imm} Let $q_{2n}^{\pm}(k)$ be the probability that an element of $\Omega^{\pm}(2n,q)$ has a $k$-dimensional fixed space. Then
$q_{2n}^{\pm}(2k+1)=0$ and \begin{eqnarray*} q_{2n}^{\pm}(2k) & = & \frac{q^k}{|GL_k(q^2)|} \sum_{j=0}^{n-k} \frac{(-1)^j}{q^{(2k-1)j} (q^{2j}-1) \cdots (q^4-1)(q^2-1)}\\ & & \pm \frac{(-1)^{n-k}}{q^{2k(n-k)} |GL_k(q^2)| (q^{2(n-k)}-1) \cdots (q^4-1)(q^2-1)}. \end{eqnarray*} \end{cor}

\subsection{Example 2: Enumeration of unipotent elements by dimension of fixed space}

The paper \cite{FG3} used generating functions to enumerate unipotent elements in orthogonal groups of even characteristic, showing
that the number of unipotent elements of $O^{\pm}_{2n}(q)$ is $q^{2n^2-2n+1} \left(1+\frac{1}{q} \mp \frac{1}{q^n} \right)$. In this example, we give a more refined count. Similar results for the finite general linear and unitary groups appear in \cite{L0}, \cite{F3}.

\begin{theorem} \label{unipdim}
\begin{enumerate}
\item The proportion of elements of $O^{\pm}_{2n}(q)$ which are unipotent and have a fixed space of dimension $2k$ is
\[ \frac{(1-1/q^{2k})(1-1/q^{2(k+1)}) \cdots (1-1/q^{2(n-1)})}{q^{n-2k}|GL_k(q^2)|(1-1/q^2)(1-1/q^4) \cdots (1-1/q^{2(n-k)})} \left[ \frac{1}{2} \pm \frac{1}{2q^n} \right]. \]
\item The proportion of elements of $O^{\pm}_{2n}(q)$ which are unipotent and have a fixed space of dimension $2k+1$ is \[ \frac{1}{2 \cdot q^{n-1}|GL_k(q^2)|} \frac{(1-1/q^{2(k+1)})(1-1/q^{2(k+2)}) \cdots (1-1/q^{2(n-1)})}{(1-1/q^2)(1-1/q^4) \cdots (1-1/q^{2(n-k-1)})}.\]
\end{enumerate}
\end{theorem}

\begin{proof} Let $u_{2n}^{\pm}(2k,q)$ denote the proportion of elements of $O^{\pm}_{2n}(q)$ which are unipotent and have a fixed space of dimension $2k$. For part 1 of the theorem, arguing as in the proof of Theorem \ref{fixdim} (and noting that all partitions coming from polynomials other than $z-1$ are empty) gives that $u_{2n}^+(2k)+u_{2n}^-(2k)$ is the
coefficient of $u^{2n}$ in \[ \frac{u^{2k} q^k}{q^{2k^2}(1-u^2/q)(1-1/q^2) \cdots (1-u^2/q^{2k-1})(1-1/q^{2k})}.\] This is equal to $\frac{q^n}{|GL_k(q^2)|}$ multiplied by the coefficient of $u^{n-k}$ in \[ \frac{1}{(1-u/q^2)(1-u/q^4) \cdots (1-u/q^{2k})}.\] Applying Theorem 349 of \cite{HW}, one concludes that \begin{eqnarray*} & & u_{2n}^+(2k)+u_{2n}^-(2k)\\ & = & \frac{1}{q^{n-2k}|GL_k(q^2)|} \frac{(1-1/q^{2k})(1-1/q^{2(k+1)}) \cdots (1-1/q^{2(n-1)})}{(1-1/q^2)(1-1/q^4) \cdots (1-1/q^{2(n-k)})}. \end{eqnarray*}

By part 2 of Theorem \ref{con} and the remark following it, $u_{2n}^+(2k)-u_{2n}^-(2k)$ is the proportion of elements of $GL_n(q^2)$ which are unipotent and have a fixed space of dimension $k$. This is known (\cite{L0}, \cite{F3}) to be \[ \frac{1}{q^{2(n-k)} |GL_k(q^2)|}  \frac{(1-1/q^{2k})(1-1/q^{2(k+1)}) \cdots (1-1/q^{2(n-1)})}{(1-1/q^2)(1-1/q^4) \cdots (1-1/q^{2(n-k)})},\] and part 1 of the theorem follows.

For part 2 of the theorem, arguing as in the proof of part 2 of Theorem \ref{fixdim} (again noting that all partitions coming from polynomials other than $z-1$ are empty), one obtains that $u_{2n}^+(2k+1)+u_{2n}^-(2k+1)$ is the
coefficient of $u^{2n}$ in \[ \frac{u^{2k+2}}{q^{2k^2+k}(1-u^2/q)(1-1/q^2) \cdots (1-1/q^{2k})(1-u^2/q^{2k+1})}.\] This is equal to $\frac{1}{q^{k}|GL_k(q^2)|}$ multiplied by the coefficient of $u^{n-k-1}$ in \[ \frac{1}{(1-u/q)(1-u/q^3) \cdots (1-u/q^{2k+1})}.\] Applying Theorem 349 of \cite{HW}, one concludes that \begin{eqnarray*} & & u_{2n}^+(2k+1)+u_{2n}^-(2k+1) \\ & = & \frac{1}{q^{n-1}|GL_k(q^2)|} \frac{(1-1/q^{2(k+1)}) (1-1/q^{2(k+2)})\cdots (1-1/q^{2(n-1)})}{(1-1/q^2) (1-1/q^4) \cdots (1-1/q^{2(n-k-1)})}.\end{eqnarray*} By part 2 of Theorem \ref{con}, $u_{2n}^+(2k+1)=u_{2n}^-(2k+1)$, and the result follows. \end{proof}

Arguing as in the proof of Corollary \ref{imm}, the following result is immediate.

\begin{cor} Let $v_{2n}^{\pm}(k,q)$ denote the proportion of elements of $\Omega^{\pm}_{2n}(q)$ which are unipotent and have a fixed space of dimension $k$. Then $v_{2n}^{\pm}(k,q)=0$ if $k$ is odd, and $v_{2n}^{\pm}(2k,q)$ is equal to
\[ \frac{1}{q^{n-1}|GL_k(q^2)|} \frac{(1-1/q^{2(k+1)})(1-1/q^{2(k+2)}) \cdots (1-1/q^{2(n-1)})}{(1-1/q^2)(1-1/q^4) \cdots (1-1/q^{2(n-k-1)})}.\] \end{cor}

\subsection{Example 3: Cyclic matrices}

A matrix over $\mathbb{F}_q$ is called cyclic if its characteristic polynomial is equal to its minimal polynomial. Motivated by applications to computational group theory \cite{NP1}, the proportion of cyclic matrices in $O^{\pm}_{2n}(q)$ (even characteristic included) has been studied in \cite{NP2} via geometric techniques and in \cite{FNP} via generating functions. Let us illustrate how Theorem \ref{cycindex} reproduces the generating functions of \cite{FNP}.

Let $c_O^{\pm}(2n,q)$ denote the proportion of cyclic matrices in the even characteristic orthogonal groups $O^{\pm}_{2n}(q)$. Define generating functions \[ C_{O^+}(u) = 1 + \sum_{n \geq 1} c_{O^+}(2n,q) u^n; \ \ \ C_{O^-}(u) = \sum_{n \geq 1} c_{O^-}(2n,q) u^n .\] An element of $O^{\pm}_{2n}(q)$ is cyclic if and only if all the $\lambda_{\phi}$ appearing in its rational canonical form have at most 1 part. Thus part 1 of Theorem \ref{cycindex} implies that \begin{eqnarray*} & & C_{O^+}(u) + C_{O^-}(u) \\ & = & \left( 1+ \frac{u}{1-u/q} \right) \prod_{d \geq 1} \left( 1+ \frac{u^d}{(q^d+1)(1-u^d/q^d)} \right)^{N^*(q;2d)}\\ & & \cdot \prod_{d \geq 1} \left( 1+\frac{u^d}{(q^d-1)(1-u^d/q^d)} \right)^{M^*(q;d)}. \end{eqnarray*} Here $N^*(q;d)$ denotes the number of monic irreducible self-conjugate polynomials $\phi$ of degree $d$ over $\mathbb{F}_q$ and $M^*(q;d)$ denotes the number of (unordered) monic irreducible conjugate pairs $\{\phi,\phi^* \}$ of degree $d$ over $\mathbb{F}_q$.
Part 2 of Theorem \ref{cycindex} implies that \begin{eqnarray*} C_{O^+}(u) - C_{O^-}(u) & = &  \prod_{d \geq 1} \left( 1- \frac{u^d}{(q^d+1)(1+u^d/q^d)} \right)^{N^*(q;2d)}\\ & & \cdot \prod_{d \geq 1} \left( 1+\frac{u^d}{(q^d-1)(1-u^d/q^d)} \right)^{M^*(q;d)}. \end{eqnarray*}

Similarly, Theorem \ref{cycindex} reproduces the generating functions of \cite{FNP} for the proportions of separable (square-free characteristic polynomial) and semisimple (square-free minimal polynomial) matrices in $O^{\pm}_{2n}(q)$. The proportions of regular semisimple elements in $O^{\pm}_{2n}(q)$ and $\Omega^{\pm}_{2n}(q)$ were studied by generating functions in \cite{FG4}, and Theorem \ref{cycindex} captures those results too.

\section{Random partitions} \label{partitions}

In this section we use our results about the even characteristic orthogonal groups to define and study a probability measure $R_{(u,q)}$ on the set of all partitions $\lambda$ of all natural numbers such that all odd parts of $\lambda$ occur with even multiplicity. We also study related measures $R^{e}_{(u,q)}$, $R^{o}_{(u,q)}$ arising from the index two simple subgroup $\Omega^{\pm}(2n,q)$ of $O^{\pm}(2n,q)$ and its non-trivial coset respectively.

These random partitions are very natural objects (analogous to those defined in \cite{F3}, \cite{F2} for the other classical groups). One reason to be interested in these measures is that they can be used to give probabilistic proofs of Theorems \ref{fixdim} and \ref{unipdim} (arguing along the lines of \cite{F2}). We also mention that the corresponding measures for the general linear groups arise in the Cohen-Lenstra \cite{CL} heuristics for number fields (the thesis \cite{Le} discusses this), and we have high hopes that the measures $R_{(u,q)}$ will arise in a number-theoretic context too.

\begin{definition}
{\rm Fix $0<u<q^{1/2}$ and $q$ a prime power. The measure $R_{(u,q)}$ is defined on the set of all partitions $\lambda$ (the size can vary) such that all odd parts occur with even multiplicity, by the formula:
\begin{eqnarray*} R_{(u,q)}(\lambda) & = & \frac{\prod_{i \geq 1} (1-u^2/q^{2i-1})}{1+u^2} \\ & & \cdot \frac{q^{l(\lambda)} u^{|\lambda|
}}{q^{n(\lambda)+\frac{|\lambda|}{2}+\frac{o(\lambda)}{2}} \prod_i (1-1/q^2) (1-1/q^4) \cdots (1-1/q^{2 \lfloor m_i(\lambda)/2 \rfloor})}.\end{eqnarray*}
}
\end{definition}

Theorem \ref{relate} relates the measure $R_{(u,q)}$ to the asymptotics of finite orthogonal groups. The use of auxiliary randomization (i.e. randomizing the variable $n$) is a mainstay of statistical mechanics known as the grand canonical ensemble. We say that an infinite collection of random variables is independent if any finite subcollection is.

\begin{theorem} \label{relate}
\begin{enumerate}
\item Fix $u$ with $0<u<1$. Then choose a random even natural number $N$ such that the probability that $N=0$ is $\frac{1-u^2}{1+u^2}$ and the probability that $N=2n \geq 2$ is equal to $2u^{2n} \frac{1-u^2}{1+u^2}$. Choose one of $O^{\pm}(N,q)$ at random (each with probability 1/2), and let $g$ be a random element of the chosen group. Let $\Lambda_{\phi}(g)$ be the partition corresponding to the polynomial $\phi$ in the rational canonical form of $g$. Then as $\phi$ varies, aside from the fact that $\Lambda_{\phi}=\Lambda_{\phi^*}$, these random variables are independent with probability laws the same as for the symplectic groups in Theorem 1 of \cite{F2}, except for the polynomial $z-1$ which has the distribution $R_{(u,q)}$.

\item Choose one of $O^{\pm}_{2n}(q)$ at random (each with probability 1/2), and let $g$ be a random element of the chosen group. Let $\Lambda_{\phi}(g)$ be the partition corresponding to the polynomial $\phi$ in the rational canonical form of $g$. Let $q$ be fixed and $n \rightarrow \infty$. Then as $\phi$ varies, aside from the fact that $\Lambda_{\phi}=\Lambda_{\phi^*}$, these random variables are independent with probability laws the same as for the symplectic groups in Theorem 1 of \cite{F2}, except for the polynomial $z-1$ which has the distribution $R_{(1,q)}$.
\end{enumerate}
\end{theorem}

\begin{proof} The method of proof is analogous to that used for the other classical groups (see the survey \cite{F0}), so we demonstrate the claim for $\Lambda_{z-1}$, as that is the interesting new feature. In part 1 of Theorem \ref{cycindex}, set $x_{\phi,\lambda}=1$ for $\phi \neq z-1$ and $x_{z-1,\lambda}=x_{z-1,\lambda} \cdot u^{|\lambda|}$. One obtains the equation

\item \begin{eqnarray*} & & 1+\sum_{n \geq 1}
\frac{u^{2n}}{|O^+_{2n}(q)|} \sum_{g \in O^+_{2n}(q)}
x_{z-1,\lambda_{z-1}(g)}  u^{|\lambda_{z-1}(g)|} \\ & & + \sum_{n \geq 1}
\frac{u^{2n}}{|O^-_{2n}(q)|} \sum_{g \in O^-_{2n}(q)}
x_{z-1,\lambda_{z-1}(g)} u^{|\lambda_{z-1}(g)|}\\
& = & \left( \sum_{|\lambda| \
even \atop i \ odd \Rightarrow m_i \ even} \frac{ x_{z-1,\lambda}
u^{|\lambda|} q^{l(\lambda)} }{q^{n(\lambda)+\frac{|\lambda|}{2}+\frac{o(\lambda)}{2}}
\prod_i (1-1/q^2) \cdots (1-1/q^{2 \lfloor \frac{m_i(\lambda)}{2}
\rfloor})} \right)\\
& & \cdot \prod_{\phi = {\phi^*} \atop \phi
\neq z-1} \left( \sum_{\lambda} \frac{
u^{|\lambda| \cdot \deg(\phi)}}{B(\phi,\lambda)} \right)
\prod_{\{\phi,\phi^*\} \atop \phi \neq \phi^*} \left( \sum_{\lambda}
\frac{u^{2|\lambda| \cdot
\deg(\phi)}}{B(\phi,\lambda) B(\phi^*,\lambda)} \right)\\
& = & \left[ \frac{\prod_{i \geq 1} (1-u^2/q^{2i-1})}{1-u^2} \right] \\
& & \cdot \left( \sum_{|\lambda| \
even \atop i \ odd \Rightarrow m_i \ even} \frac{ x_{z-1,\lambda}
u^{|\lambda|} q^{l(\lambda)} }{q^{n(\lambda)+\frac{|\lambda|}{2}+\frac{o(\lambda)}{2}}
\prod_i (1-1/q^2) \cdots (1-1/q^{2 \lfloor \frac{m_i(\lambda)}{2}
\rfloor})} \right).
\end{eqnarray*} The final equality follows as in the proof of part 1 of Theorem \ref{fixdim}.
Multiplying both sides by $(1-u^2)/(1+u^2)$ implies that
\begin{eqnarray*} & & \frac{1-u^2}{1+u^2} +\sum_{n \geq 1}
\frac{2(1-u^2)u^{2n}}{1+u^2} \left[ \frac{\sum_{g \in O^+_{2n}(q)}
x_{z-1,\lambda_{z-1}(g)} u^{|\lambda_{z-1}(g)|}}{2|O^+_{2n}(q)|} \right] \\ & & +
\sum_{n \geq 1}
\frac{2(1-u^2)u^{2n}}{1+u^2} \left[ \frac{\sum_{g \in O^-_{2n}(q)}
x_{z-1,\lambda_{z-1}(g)}u^{|\lambda_{z-1}(g)|}}{2|O^-_{2n}(q)|} \right] \\
& = & \frac{\prod_{i \geq 1} (1-u^2/q^{2i-1})}{1+u^2} \\
& & \cdot \sum_{ |\lambda| \
even \atop i \ odd \Rightarrow m_i \ even}  \frac{q^{l(\lambda)} u^{|\lambda|} x_{z-1,\lambda}}
{q^{n(\lambda)+ \frac{|\lambda|}{2}+\frac{o(\lambda)}{2}} \prod_i (1-1/q^2)  \cdots (1-1/q^{2 \lfloor m_i(\lambda)/2 \rfloor})} \\
& = & \sum_{\lambda} R_{(u,q)}(\lambda) x_{z-1,\lambda} , \end{eqnarray*} which proves part 1.

To prove the second assertion, one uses the fact that if a Taylor series of a function $f(u^2)$ around
$0$ converges at $u=1$, then the $n \rightarrow \infty$ limit of the coefficient of $u^{2n}$ in $\frac{f(u^2)(1+u^2)}{2(1-u^2)}$
is equal to $f(1)$. \end{proof}

Next, we give a Markov chain method for sampling from the distribution $R_{(u,q)}$. Define two Markov
chains $K_1,K_2$ on the natural numbers with transition probabilities

\[ K_1(a,b) =             \left\{ \begin{array}{ll}
        \frac{u^aP'_{O,u}(b)}{P'_{Sp,u}(a) q^{\frac{a^2-b^2+2(a+1)b}{4}} (q^{a-b
}-1) \cdots (q^4-1)(q^2-1)}             & \mbox{if $a-b$ even, $b \leq a$}\\

                0       & \mbox{else}

                                                \end{array}
                        \right.                  \]

\[ K_2(a,b) =  \left\{ \begin{array}{ll}
        \frac{u^aP'_{Sp,u}(b) q^{(a-b)^2/4}}{P'_{O,u}(a) q^{\frac{a^2+b}{2}-a} (
q^{a-b}-1) \cdots (q^4-1)(q^2-1)}             & \mbox{if $a-b$ even, $b \leq a$}\\

        \frac{u^aP'_{Sp,u}(b) q^{((a-b)^2-1)/4}}{P'_{O,u}(a) q^{\frac{a^2-a}{2}}
 (q^{a-b-1}-1) \cdots (q^4-1)(q^2-1)}    & \mbox{if
$a-b$ odd, $b \leq a$}\\
0 & \mbox{else}

                        \end{array}
                        \right.  \]

where $P'_{Sp,u},P'_{O,u}$ are defined as follows: \begin{eqnarray*}
P'_{Sp,u}(2k) & = & \frac{u^{2k}}{q^{2k^2+k}(1-u^2/q)(1-1/q^2) \cdots (1-u^2/q^{
2k-1})(1-1/q^{2k})}\\
P'_{Sp,u}(2k+1) & = & \frac{u^{2k+2}}{q^{2k^2+3k+1}(1-u^2/q)(1-1/q^2) \cdots (1-
1/q^{2k})(1-u^2/q^{2k+1})}\\
P'_{O,u}(2k) & = & \frac{u^{2k}}{q^{2k^2-k}(1-u^2/q)(1-1/q^2) \cdots (1-u^2/q^{2
k-1})(1-1/q^{2k})}\\
P'_{O,u}(2k+1) & = & \frac{u^{2k+1}}{q^{2k^2+k}(1-u^2/q)(1-1/q^2) \cdots (1-1/q^
{2k})(1-u^2/q^{2k+1})}.\\
\end{eqnarray*}

\begin{theorem} \label{samp} Let $\lambda_1'$ be a random natural number which is equal to $2k$ with probability \[ \frac{\prod_{i=1}^{\infty} (1-u^2/q^{2i-1})}{1+u^2} \frac{u^{2k}
}{q^{2k^2-k}(1-u^2/q)(1-1/q^2) \cdots (1-u^2/q^{2k-1})(1-1/q^{2k})} \] and equal to $2k+1$ with probability \[  \frac{\prod_{i=1}^{\infty} (1-u^2/q^{2i-1})}{1+u^2} \frac{u^{2
k+2}}{q^{2k^2+k}(1-u^2/q)(1-1/q^2) \cdots (1-1/q^{2k})(1-u^2/q^{2k+1})}\] Define $\lambda_2',\lambda_3',\cdots$ according to the rules
that if $\lambda_i'=a$, then $\lambda_{i+1}'=b$ with probability $K_1(a,b)$ if
$i$ is odd and probability $K_2(a,b)$ if $i$ is even. Then the resulting
partition is distributed according to $R_{(u,q)}$.
\end{theorem}

\begin{proof} The crucial observation is that $R_{(u,q)}$ can be related to a measure $P_{Sp,u}$ studied in \cite{F2}. Indeed, comparing formulas one sees that \[ R_{(u,q)}(\lambda) = \frac{q^{l(\lambda)}}{(1+u^2)} \cdot P_{Sp,u}(\lambda) \] for all $\lambda$. Hence the theorem follows from Theorems 3 and 4 of \cite{F2}.
\end{proof}

Next, we define and study measures $R^e_{(u,q)}$ and $R^o_{(u,q)}$.

\begin{definition}
{\rm Fix $0<u<q^{1/2}$ and $q$ a prime power. The measure $R^e_{(u,q)}$ is defined on the set of all partitions $\lambda$ (the size can vary) with an even number of parts and such that all odd parts occur with even multiplicity, by the formula:
\begin{eqnarray*} R^e_{(u,q)}(\lambda) & = & \prod_{i \geq 1} (1-u^2/q^{2i-1}) \\ & & \cdot \frac{q^{l(\lambda)} u^{|\lambda|
}}{q^{n(\lambda)+\frac{|\lambda|}{2}+\frac{o(\lambda)}{2}} \prod_i (1-1/q^2) (1-1/q^4) \cdots (1-1/q^{2 \lfloor m_i(\lambda)/2 \rfloor})}.
\end{eqnarray*} We also define a measure $R^o_{(u,q)}$ on the set of all partitions $\lambda$ (the size can vary) with an odd number of parts and such that all odd parts occur with even multiplicity, by the formula:
\begin{eqnarray*}R^o_{(u,q)}(\lambda) & = & \frac{1}{u^2} \prod_{i \geq 1} (1-u^2/q^{2i-1}) \\ & & \cdot \frac{q^{l(\lambda)} u^{|\lambda|
}}{q^{n(\lambda)+\frac{|\lambda|}{2}+\frac{o(\lambda)}{2}} \prod_i (1-1/q^2) (1-1/q^4) \cdots (1-1/q^{2 \lfloor m_i(\lambda)/2 \rfloor})}.
\end{eqnarray*}
}
\end{definition}

These measures arise from $\Omega^{\pm}$ and its non-trivial coset in the following way. We omit the proof, which is almost identical to that of Theorem \ref{relate}.

\begin{theorem} \label{relate1}
\begin{enumerate}
\item Fix $u$ with $0<u<1$. Then choose a random even natural number $N$ such that the probability that $N=2n$ is equal to $(1-u^2)u^{2n}$. Choose one of $\Omega^{\pm}_N(q)$ at random (each with probability 1/2), and let $g$ be a random element of the chosen group. Let $\Lambda_{\phi}(g)$ be the partition corresponding to the polynomial $\phi$ in the rational canonical form of $g$. Then as $\phi$ varies, aside from the fact that $\Lambda_{\phi}=\Lambda_{\phi^*}$, these random variables are independent with probability laws the same as for the symplectic groups in Theorem 1 of \cite{F2}, except for the polynomial $z-1$ which has the distribution $R^{e}_{(u,q)}$.

\item Choose one of $\Omega^{\pm}_{2n}(q)$ at random (each with probability 1/2), and let $g$ be a random element of the chosen group. Let $\Lambda_{\phi}(g)$ be the partition corresponding to the polynomial $\phi$ in the rational canonical form of $g$. Let $q$ be fixed and $n \rightarrow \infty$. Then as $\phi$ varies, aside from the fact that $\Lambda_{\phi}=\Lambda_{\phi^*}$, these random variables are independent with probability laws the same as for the symplectic groups in Theorem 1 of \cite{F2}, except for the polynomial $z-1$ which has the distribution $R^{e}_{(1,q)}$.

\item Fix $u$ with $0<u<1$. Then choose a random even natural number $N$ such that the probability that $N=2n \geq 2$ is equal to $(1-u^2)u^{2(n-1)}$. Choose one of the non-trivial cosets of $\Omega^{\pm}_N(q)$ at random (each with probability 1/2), and let $g$ be a random element of the chosen coset. Let $\Lambda_{\phi}(g)$ be the partition corresponding to the polynomial $\phi$ in the rational canonical form of $g$. Then as $\phi$ varies, aside from the fact that $\Lambda_{\phi}=\Lambda_{\phi^*}$, these random variables are independent with probability laws the same as for the symplectic groups in Theorem 1 of \cite{F2}, except for the polynomial $z-1$ which has the distribution $R^{o}_{(u,q)}$.

\item Choose one of the non-trivial cosets of $\Omega^{\pm}_{2n}(q)$ at random (each with probability 1/2), and let $g$ be a random element of the chosen coset. Let $\Lambda_{\phi}(g)$ be the partition corresponding to the polynomial $\phi$ in the rational canonical form of $g$. Let $q$ be fixed and $n \rightarrow \infty$. Then as $\phi$ varies, aside from the fact that $\Lambda_{\phi}=\Lambda_{\phi^*}$, these random variables are independent with probability laws the same as for the symplectic groups in Theorem 1 of \cite{F2}, except for the polynomial $z-1$ which has the distribution $R^{o}_{(1,q)}$.
\end{enumerate}
\end{theorem}

Finally, we describe an algorithm for sampling from $R^e_{(u,q)}$ and $R^o_{(u,q)}$, which is proved along the same lines as Theorem \ref{samp}.

\begin{theorem}
\begin{enumerate}
\item Let $\lambda_1'$ be a random even natural number which is equal to $2k$ with probability \[ \prod_{i=1}^{\infty} (1-u^2/q^{2i-1}) \frac{u^{2k}
}{q^{2k^2-k}(1-u^2/q)(1-1/q^2) \cdots (1-u^2/q^{2k-1})(1-1/q^{2k})}. \] Define $\lambda_2',\lambda_3',\cdots$ according to the rules
that if $\lambda_i'=a$, then $\lambda_{i+1}'=b$ with probability $K_1(a,b)$ if
$i$ is odd and probability $K_2(a,b)$ if $i$ is even. Then the resulting
partition is distributed according to $R^e_{(u,q)}$.

\item Let $\lambda_1'$ be a random odd natural number which is equal to $2k+1$ with probability \[  \prod_{i=1}^{\infty} (1-u^2/q^{2i-1}) \frac{u^{2
k}}{q^{2k^2+k}(1-u^2/q)(1-1/q^2) \cdots (1-1/q^{2k})(1-u^2/q^{2k+1})}\] Define $\lambda_2',\lambda_3',\cdots$ according to the rules
that if $\lambda_i'=a$, then $\lambda_{i+1}'=b$ with probability $K_1(a,b)$ if
$i$ is odd and probability $K_2(a,b)$ if $i$ is even. Then the resulting
partition is distributed according to $R^o_{(u,q)}$.
\end{enumerate}
\end{theorem}

\end{document}